\documentclass[11pt,reqno]{amsart}

\usepackage[T1]{fontenc}
\usepackage{lmodern}
\usepackage{microtype}
\usepackage{amsmath,amssymb,amsthm,mathtools,mathrsfs}
\usepackage{enumitem}
\usepackage{booktabs,longtable,array}
\usepackage[hidelinks]{hyperref}

\usepackage{orcidlink}

\newtheorem{theorem}{Theorem}[section]
\newtheorem{lemma}[theorem]{Lemma}
\newtheorem{proposition}[theorem]{Proposition}
\newtheorem{corollary}[theorem]{Corollary}

\theoremstyle{definition}

\theoremstyle{remark}
\newtheorem{remark}[theorem]{Remark}

\theoremstyle{definition}
\newtheorem{example}[theorem]{Example}

\newcommand{\norm}[1]{\left\lVert #1\right\rVert}
\newcommand{\dd}{\,\mathrm{d}}

\title[Compactness and essential norms of bilinear Hardy operators]{Compactness and Essential Norms of Weighted Bilinear Hardy Operators on General Intervals}

\author{Saikat Kanjilal\,\orcidlink{0000-0002-4359-8343}}
\address{Department of Mathematics, JIS University,
Kolkata - 700109, West Bengal, India}

\email{saikat.kanjilal.07@gmail.com}

\date{}

\hypersetup{
  pdftitle={Compactness and Essential Norms of Weighted Bilinear Hardy Operators on General Intervals},
  pdfauthor={Saikat Kanjilal},
  pdfsubject={Compactness, essential quasi-norms, and quantitative approximation for weighted bilinear Hardy operators},
  pdfkeywords={weighted bilinear Hardy operator, quasi-Banach target, essential norm, Hausdorff measure of noncompactness, approximation numbers, endpoint localisation}
}

\begin{document}

\begin{abstract}
Let \(I=(a,b)\) be an open interval with finite or infinite endpoints.  For
\(1<p_1,p_2<\infty\) and \(0<q<\infty\), we study the weighted bilinear
Hardy operator
\(
\mathcal H_2(f,g)(x)=(\int_a^x f)(\int_a^x g)
\)
from \(L^{p_1}(v_1;I)\times L^{p_2}(v_2;I)\) to \(L^q(u;I)\), including the
quasi-Banach target range \(0<q<1\).  With
\(\vartheta=\min\{1,q\}\), the target and bounded-bilinear-map spaces are
treated through their complete powered metrics.  Under explicit local
integrability assumptions on the weights, compactness is equivalent to
boundedness and vanishing endpoint-tail quasi-norms.  Whenever the full
operator is bounded, the \(\vartheta\)-power of the directed combined-tail
limit equals exactly the powered essential distance, finite-rank distance,
and Hausdorff measure of noncompactness in the target metric.  Equivalently,
the unpowered quantities agree under the stated quasi-norm-radius
normalisation.  Piecewise-constant middle-window approximants give
quantitative bounds for the bilinear approximation numbers. The
finite-rank-distance identity then implies that these numbers converge to the
essential (quasi-)norm.  Applying published norm-equivalent boundedness
characteristics to endpoint-truncated output weights yields explicit
compactness criteria in the convex, mixed, moderately subcritical, and deeply
subcritical regimes.  Algebraic companion exponents are used to make the
negative powers arising for \(q<1\) explicit.  Four worked examples include
an asymmetric mixed-range threshold and a nonlocally convex
\(L^{1/2}(0,1)\) example for which
\(a_{m+1}^{(2)}(\mathcal H_2)=O(m^{-1/2})\).

\end{abstract}



\maketitle

\section{Introduction}
\label{sec:introduction}

Weighted Hardy inequalities are a classical model for two-weight estimates
in analysis.  The one-dimensional theory begins with the Muckenhoupt
condition and extends to alternative weight characterisations, iterated
operators, multidimensional forms, and Hardy--Steklov constructions. See
\cite{Muckenhoupt1972,PerssonSamko2024} and the references therein.

Compactness requires more than boundedness.  For linear Hardy and Volterra
operators, the established theory includes endpoint localisation,
finite-rank approximation, distance from compact operators, measures of
noncompactness, and approximation numbers
\cite{EdmundsGurkaPick1994,EdmundsStepanov1994,Opic1994,
JainGupta2003,Prokhorov2000}.
Stepanov and Ushakova obtained weight criteria and two-sided noncompactness
estimates for a two-dimensional rectangular Hardy operator
\cite{StepanovUshakova2022}, and later gave upper and lower
approximation-number estimates for the same operator
\cite{StepanovUshakova2024}.  These results provide the linear and
rectangular antecedents for the present work.

At the operator-ideal level, general multilinear approximation numbers and
measures of noncompactness, together with bilinear interpolation results in
Banach and quasi-Banach settings, are available in
\cite{FernandezMastyloDaSilva2013,MastyloSilva2018,BesoyCobos2019}.  They do
not provide the compact-middle cutoff identity isolated below or verify its
Hardy-specific local hypotheses.

The corresponding bilinear boundedness problem has a separate development.
Aguilar Ca\~nestro, Ortega Salvador and Ram\'irez Torreblanca characterised
the admissible weights for the basic bilinear Hardy inequality in the Banach
exponent range \cite{AguilarCanestroOrtegaRamirez2012}.  K\v{r}epela
subsequently gave an iteration method for the finite-exponent range with
\(q>1\) \cite{Krepela2017}.  Kanjilal, Persson and Shambilova stated the five
boundedness regimes for every \(0<q<\infty\), including the fully subcritical
cases with \(q<1\) \cite{KanjilalPerssonShambilova2019}.  Equivalent
conditions and related operators were also studied in
\cite{StepanovShambilova2019,StepanovUshakova2019,
GogatishviliJainKanjilal2022}.
The surrounding theory includes bilinear Hardy--Steklov operators,
multidimensional and discrete variants, weak-type inequalities, and
metric-measure formulations
\cite{JainKanjilalStepanovUshakova2018,JainKanjilalPersson2019,
JainKanjilalShambilovaStepanov2020,GarciaGarciaOrtegaSalvador2023,
GarciaGarciaOrtegaSalvador2024,RuzhanskyShriwastawaVerma2024,
MohantyJainJain2025}.

Let \(I=(a,b)\) be any nonempty open interval, with either endpoint allowed
to be infinite.  We study the same-variable bilinear Hardy map
\[
\mathcal H_2(f,g)(x)=Hf(x)Hg(x),
\qquad
Hf(x)=\int_a^x f(t)\,\dd t,
\]
from
\[
L^{p_1}(v_1;I)\times L^{p_2}(v_2;I)
\quad\text{to}\quad
L^q(u;I),
\qquad
1<p_1,p_2<\infty,\quad 0<q<\infty.
\]
Compactness is understood jointly on the product of the two unit balls.  It
cannot be inferred merely from compactness of the slice maps obtained by
fixing one input.

The range \(0<q<1\) requires a separate topological formulation.  The
functional on \(L^q(u;I)\) is then a quasi-norm rather than a norm, and the
space is nonlocally convex.  We put
\[
\vartheta=\min\{1,q\},
\qquad
\rho(h,k)=\norm{h-k}_{L^q(u;I)}^\vartheta,
\]
and work throughout in this complete metric.  The extension is not obtained
by invoking a Banach-space theorem formally.  Uniform absolute continuity on
compact subsets, compactness of finite-dimensional ranges, closure of compact
bilinear maps, compact-perturbation estimates, and finite-net estimates are
all proved with the \(\vartheta\)-triangle inequality.  At the same time,
the exact \(q\)-additivity of the quasi-norm on disjoint supports preserves
the sharp two-tail and finite-rank formulas.

Four concrete ingredients make this metric formulation applicable to the
operator at hand.  We prove completeness and indicator contraction in the
weighted \(L^q\) target, prove completeness of the corresponding space of
bounded bilinear maps, construct finite-rank approximants on every bounded
middle window, and give explicit shrinking endpoint sets for all four
interval geometries.  These statements supply the hypotheses used by the
localisation and finite-centre arguments rather than leaving them as abstract
metric assumptions.

After proving that every middle truncation is compact by an explicit
finite-rank construction, we separate the common cutoff argument into two
abstract tiers.  The compact-middle tier identifies compactness, essential
distance, and Hausdorff noncompactness with the directed output tail.  The
finite-rank tier adds finite-rank distance and the approximation-number limit
under a local finite-rank-density hypothesis.  The interior estimates and
partition sampler verify both tiers for \(\mathcal H_2\).  Consequently, the
full operator is compact if and only if it is bounded and the norms of its
lower and upper endpoint restrictions tend to zero.

The principal quantitative result concerns a bounded \(\mathcal H_2\).  If
\(\gamma(c,d)\) is the quasi-norm of its combined lower and upper tail outside
\((c,d)\), then
\[
\lim_{\substack{c\downarrow a\\d\uparrow b}}\gamma(c,d)
\]
equals each of the following quantities.
\[
\norm{\mathcal H_2}_{\mathrm e},\qquad
\operatorname{dist}_{\mathrm{fr}}(\mathcal H_2),\qquad
\chi_{L^q(u;I)}
\bigl(\mathcal H_2(B_{p_1}\times B_{p_2})\bigr).
\]
Thus the endpoint tails determine not only whether compactness holds, but
also the exact distance from compact and finite-rank bilinear maps under the
normalisations fixed in Section~\ref{sec:preliminaries}.  The same partition
construction yields explicit upper estimates for the bilinear approximation
numbers.  Their convergence to the essential (quasi-)norm is then an
immediate consequence of the finite-rank-distance identity and the defining
rank infima.

At the metric scale, the exact statement is obtained by raising every term
to \(\vartheta\). The powered essential distance, powered finite-rank
distance, metric-radius Hausdorff measure of noncompactness, and powered tail
limit coincide.  The quasi-norm-radius statement displayed above is its
equivalent unpowered form.

Finally, the localization principle transfers published norm-equivalent
boundedness characteristics to compactness and essential-norm criteria.  The
simple Muckenhoupt-type expression applies when
\(q\geq\max(p_1,p_2)\).  We also state the weight-only conditions in the two
mixed regimes and in both fully subcritical regimes.  In the latter case the
criterion changes at
\[
\frac1q=\frac1{p_1}+\frac1{p_2}.
\]
Combining the localization principle with those cited boundedness
characterizations gives compactness criteria for each finite-exponent regime
covered by them.  The displayed weight formulas and the exponent split are
prior boundedness results. The contribution here is their endpoint-truncated
compactness and essential-distance consequence.

Thus the abstract compact-middle theorem records the reusable finite-\(q\)
localization infrastructure.  The Hardy-specific contribution is its
constructive realization for the same-variable bilinear map. For every
bounded \(\mathcal H_2\), one combined-tail limit is identified exactly with
compact distance, finite-rank distance, and Hausdorff noncompactness,
including through the powered metric when \(0<q<1\).  The explicit weight
criteria and examples are applications and consistency checks for that
operator-theoretic core.

\subsection*{Comparison with the closest results}

The relation between the principal results of this paper and the closest
linear and bilinear work is summarised below.

{\footnotesize
\renewcommand{\arraystretch}{1.00}
\begin{longtable}{@{}
>{\raggedright\arraybackslash}p{0.20\textwidth}
>{\raggedright\arraybackslash}p{0.34\textwidth}
>{\raggedright\arraybackslash}p{0.36\textwidth}@{}}
\toprule
Source & Principal theorem & Relation to the present results\\
\midrule
\endfirsthead
\toprule
Source & Principal theorem & Relation to the present results\\
\midrule
\endhead
\bottomrule
\endfoot
Edmunds--Gurka--Pick
\cite{EdmundsGurkaPick1994}
&
Endpoint/interior compactness criteria for linear generalized Hardy maps in
weighted Banach function spaces (Theorem~1 and Lemma~2), and two-sided
finite-rank-distance estimates (Theorem~2).
&
The present results apply an analogous localization scheme to the joint
bilinear map \(Hf\,Hg\), include \(0<q<1\), and obtain an exact combined-tail
identity under the present normalization.
\\
Opic
\cite{Opic1994}
&
Lower and upper bounds for the distance of generalized weighted linear
Hardy--Riemann--Liouville operators from compact operators.
&
Theorem~\ref{thm:essential-norm} treats a bilinear product map and identifies
its compact distance exactly with one combined-tail limit.
\\
Edmunds--Stepanov
\cite{EdmundsStepanov1994}
&
Compactness, measures of noncompactness, and approximation numbers for a
class of linear Volterra integral operators, with endpoint quantities in
two-sided estimates.
&
For the bilinear product map, the exact tail identity identifies the limiting
finite-rank obstruction. The approximation-number limit then follows from
the definition.
\\
Jain--Gupta
\cite{JainGupta2003}
&
Weight characterisations for compactness of a linear Hardy--Steklov operator
with variable integration limits.
&
The present operator instead has two independent inputs and output
\(Hf\,Hg\). Compactness concerns their joint product-ball image and an exact
tail obstruction.
\\
Prokhorov
\cite{Prokhorov2000}
&
Boundedness and compactness criteria for a class of weighted linear
fractional Volterra operators, including quasi-Banach parameter ranges under
additional hypotheses.
&
The present results concern the nonfractional bilinear product \(Hf\,Hg\), all
four open-interval geometries, and an exact combined-tail identity.
\\
Aguilar Ca\~nestro et al., K\v{r}epela, and
Kanjilal--Persson--Shambilova
\cite{AguilarCanestroOrtegaRamirez2012,Krepela2017,
KanjilalPerssonShambilova2019}
&
Norm-equivalent boundedness criteria for the same-variable bilinear Hardy
inequality.  The literal five-regime statement for \(0<q<\infty\), including
fully subcritical \(q<1\), is due to Kanjilal--Persson--Shambilova.
&
Corollaries~\ref{cor:explicit-convex-compactness},
\ref{cor:mixed-compactness}, and \ref{cor:subcritical-compactness} apply
those prior characteristics to truncated output weights, yielding
compactness criteria and two-sided essential-distance estimates.
\\
Stepanov--Ushakova
\cite{StepanovUshakova2022}
&
For a linear rectangular Hardy operator on \(\mathbb R_+^2\), compactness
criteria and a two-sided estimate for a finite-rank-distance
noncompactness quantity.
&
The present one-dimensional map is genuinely bilinear. For its combined
output tail, Theorem~\ref{thm:essential-norm} gives equality rather than norm
equivalence, including \(q<1\).
\\
Fernandez--Mastylo--da Silva, Mastylo--Silva, and Besoy--Cobos
\cite{FernandezMastyloDaSilva2013,MastyloSilva2018,BesoyCobos2019}
&
Multilinear approximation numbers under the span-of-image rank convention,
multilinear noncompactness quantities, and bilinear interpolation of
noncompactness in Banach and quasi-Banach settings.
&
These works supply operator-ideal definitions and interpolation results, but
not the compact-middle cutoff theorem or its Hardy realization.  In the
present paper, the essential-distance and noncompactness identities follow
from the compact-middle tier. Equality with finite-rank distance for
\(\mathcal H_2\) additionally uses the middle-window bilinear sampler.
\\
Stepanov--Ushakova
\cite{StepanovUshakova2024}
&
Separate upper and lower approximation-number estimates for a compact linear
rectangular Hardy operator with factorable output weight.
&
Proposition~\ref{prop:quantitative-finite-rank} constructs rank-controlled
approximants for \(Hf\,Hg\) on a middle interval.  The limiting statement is
an immediate corollary of the finite-rank-distance identity.
\\
\end{longtable}
}

The comparison also separates exact and norm-equivalent statements.  The
operator-theoretic identity in Theorem~\ref{thm:essential-norm} is exact.
The published bilinear weight characteristics cited above are equivalent to
the operator (quasi-)norm only up to exponent-dependent constants.  Accordingly,
Section~\ref{sec:characteristics} gives exact compactness criteria but
two-sided, rather than exact, weight estimates for the essential
(quasi-)norm.

The proof of the exact formula has two complementary parts.  Compact middle
truncations provide the upper estimates for essential distance and
noncompactness, while their finite-rank approximants add the finite-rank
distance.  Conversely, the image of any compact bilinear map has uniformly
vanishing \(L^q\)-mass on shrinking endpoint sets.  This forces every compact
perturbation to retain at least the limiting combined-tail quasi-norm.  The
same finite-net argument, with all distances raised to \(\vartheta\), gives
the lower bound for the Hausdorff measure of noncompactness.

Four examples illustrate the criteria.  A symmetric half-line model has a
critical weight at which boundedness holds but compactness fails.  A
unit-weight bounded-interval model has two explicitly vanishing endpoint
characteristics.  The third example uses
\[
p_1=q=2<p_2=4
\]
and produces a different sharp threshold in an asymmetric mixed exponent
regime outside the convex boundedness range.
The fourth has \(p_1=p_2=2\) and \(q=1/2\).  It verifies the negative
algebraic powers in the deep-subcritical characteristic, proves both endpoint
tails directly, and yields the quantitative estimate
\(a_{m+1}^{(2)}(\mathcal H_2)=O(m^{-1/2})\).

The input endpoints \(p_i=1\) and \(p_i=\infty\), and the target endpoint
\(q=\infty\), are not included.  Section~\ref{sec:limitations} records why
the corresponding mechanisms differ.

The paper is organised as follows.  Sections~2 and 3 set the weighted and
localisation frameworks.  Section~4 proves middle compactness and the
quantitative finite-rank estimates.  Section~5 establishes the compactness,
essential-norm, and noncompactness theorems by first isolating the abstract
two-tier localization engine and then applying it to \(\mathcal H_2\).
Section~6 records the four interval geometries.  Section~7 gives the explicit
weight criteria and the four examples.  Section~8 discusses the excluded
endpoint regimes, and Section~9 concludes.  The auxiliary compactness
principles are proved in Appendix~A.

\section{Setting and quasi-Banach preliminaries}
\label{sec:preliminaries}

Throughout this section, let
\[
-\infty\leq a<b\leq\infty,
\qquad
I:=(a,b),
\]
and let
\begin{equation}
1<p_{1},p_{2}<\infty,
\qquad
0<q<\infty.
\label{eq:standing-exponent-range}
\end{equation}
For each \(i\in\{1,2\}\), let \(p_i'\) denote the conjugate exponent of
\(p_i\), that is,
\[
\frac{1}{p_i}+\frac{1}{p_i'}=1.
\]

Let \(v_1,v_2\) be measurable input weights on \(I\), positive and finite
almost everywhere. Let \(u\) be a measurable nonnegative output weight on
\(I\).  Assume that
\[
u\in L^1_{\mathrm{loc}}(I)
\]
and
\begin{equation}
V_i(x)
:=
\int_a^x v_i(t)^{1-p_i'}\,dt
<\infty,
\qquad x\in I,\quad i=1,2.
\label{eq:local-weight-condition}
\end{equation}
Because the integrand defining \(V_i\) is positive almost everywhere,
\begin{equation}
0<V_i(x)<\infty,
\qquad x\in I.
\label{eq:Vi-strictly-positive}
\end{equation}
This observation will be relevant when negative powers of \(V_i\) occur in
the quasi-Banach weight characteristics.

Set
\[
X_i:=L^{p_i}(v_i;I),
\qquad
Y:=L^q(u;I),
\]
\[
\dd\mu_u(x):=u(x)\,\dd x.
\]
The input norms and output quasi-norm are
\[
\norm{f}_{X_i}
=
\left(\int_I |f(x)|^{p_i}v_i(x)\,dx\right)^{1/p_i}
\]
and
\[
\norm{h}_{Y}
=
\left(\int_I |h(x)|^q u(x)\,dx\right)^{1/q}.
\]
As usual, functions equal \(u(x)\,dx\)-almost everywhere are identified.

For \(f\in X_i\), define the one-sided Hardy operator by
\[
Hf(x):=\int_a^x f(t)\,dt,
\qquad x\in I.
\]
The associated bilinear Hardy operator is
\[
\mathcal H_2(f,g)(x):=Hf(x)Hg(x).
\]
Condition \eqref{eq:local-weight-condition} guarantees that \(Hf(x)\) and
\(Hg(x)\) are finite for every \(x\in I\).

\begin{remark}[Output-weight hypothesis]
\label{rem:u-hypothesis}
The assumption \(u\in L^1_{\mathrm{loc}}(I)\) is used only through the
local finiteness of the weighted output measure on bounded interior windows.
No almost-everywhere positivity assumption on \(u\) is required.
\end{remark}

\subsection{\texorpdfstring{The target topology for all finite \(q\)}
{The target topology for all finite q}}

Put
\begin{equation}
\vartheta:=\min\{1,q\}.
\label{eq:theta-definition}
\end{equation}
The space \(Y\) is a Banach space when \(q\geq1\) and a complete
quasi-Banach space when \(0<q<1\).  In both cases,
\begin{equation}
\norm{h+k}_Y^{\vartheta}
\leq
\norm{h}_Y^{\vartheta}+\norm{k}_Y^{\vartheta},
\qquad h,k\in Y.
\label{eq:theta-triangle}
\end{equation}
For \(0<q<1\), this follows from
\(|s+t|^q\leq |s|^q+|t|^q\). For \(q\geq1\), it is the ordinary triangle
inequality.  Moreover, if \(h\) and \(k\) have disjoint measurable supports,
then the stronger identity
\begin{equation}
\norm{h+k}_Y^q
=
\norm{h}_Y^q+\norm{k}_Y^q
\label{eq:disjoint-q-additivity}
\end{equation}
holds for every \(0<q<\infty\).

The formula
\begin{equation}
\rho_Y(h,k):=\norm{h-k}_Y^{\vartheta}
\label{eq:Lq-metric}
\end{equation}
defines a complete metric that induces the quasi-norm topology of \(Y\).
All closures, compactness statements, and finite-net arguments below refer
to this topology.  Thus no local convexity or Hahn--Banach argument is used
when \(q<1\).

\begin{lemma}[Weighted \(L^q\) metric and measurable cut-offs]
\label{lem:weighted-Lq-completeness}
Let \(0<q<\infty\), let \(Y=L^q(\mu_u)\), and let
\(\vartheta=\min\{1,q\}\).  Then \(\rho_Y\) in
\eqref{eq:Lq-metric} is a complete metric.  If \(E\subset I\) is measurable
and \(P_Eh:=\chi_Eh\), then
\begin{equation}
\rho_Y(P_Eh,P_Ek)\leq \rho_Y(h,k),
\qquad h,k\in Y.
\label{eq:indicator-nonexpansive}
\end{equation}
If \((E_n)\) is a sequence of measurable sets such that
\(\chi_{E_n}\to0\) \(\mu_u\)-almost everywhere, then
\begin{equation}
\rho_Y(P_{E_n}h,0)\longrightarrow0
\qquad (h\in Y).
\label{eq:indicator-tail-pointwise}
\end{equation}
\end{lemma}

\begin{proof}
The metric axioms, except completeness, follow from
\eqref{eq:theta-triangle}. Separation is understood modulo
\(\mu_u\)-almost-everywhere equality.  For \(q\geq1\), completeness is the
usual completeness of \(L^q(\mu_u)\).  We give the argument needed below for
\(0<q<1\).  In this case
\[
\rho_Y(h,k)=\int_I |h-k|^q\,\dd\mu_u.
\]
Let \((h_n)\) be \(\rho_Y\)-Cauchy.  Choose measurable representatives of
the sequence.  All subsequent pointwise statements are made on the
intersection of the countably many full-measure sets on which the relevant
identities hold.  Choose a subsequence \((h_{n_j})\) such
that
\[
\rho_Y(h_{n_{j+1}},h_{n_j})\leq 2^{-j},
\qquad j\geq1,
\]
and put \(g_j=h_{n_{j+1}}-h_{n_j}\).  Tonelli's theorem gives
\[
\int_I\sum_{j=1}^{\infty}|g_j(x)|^q\,\dd\mu_u(x)
=
\sum_{j=1}^{\infty}\rho_Y(h_{n_{j+1}},h_{n_j})<\infty.
\]
Consequently \(\sum_j|g_j(x)|^q<\infty\) almost everywhere.  At such a point
only finitely many \(|g_j(x)|\) exceed one, and thereafter
\(|g_j(x)|\leq |g_j(x)|^q\).  Hence \(\sum_jg_j(x)\) converges absolutely.
Define
\[
h=h_{n_1}+\sum_{j=1}^{\infty}g_j
\]
on this full-measure set, and define it to be zero on the complementary
null set.  The resulting function \(h\) is measurable as the pointwise limit
of measurable partial sums.  Since \(|\sum_j z_j|^q\leq\sum_j|z_j|^q\),
\[
\rho_Y(h,h_{n_J})
\leq
\sum_{j=J}^{\infty}\rho_Y(h_{n_{j+1}},h_{n_j})
\longrightarrow0.
\]
The same estimate, together with \(h_{n_1}\in L^q(\mu_u)\), shows that
\(h\in L^q(\mu_u)\).  Thus the selected subsequence converges, and the
Cauchy property gives convergence of the full sequence.

For a measurable \(E\),
\[
\norm{P_E(h-k)}_Y^q
=\int_E|h-k|^q\,\dd\mu_u
\leq\norm{h-k}_Y^q,
\]
which proves \eqref{eq:indicator-nonexpansive}.  Finally,
\[
\norm{P_{E_n}h}_Y^q
=\int_I\chi_{E_n}|h|^q\,\dd\mu_u\longrightarrow0
\]
by dominated convergence.  This implies
\eqref{eq:indicator-tail-pointwise} for every finite \(q\).
\end{proof}

A bilinear map \(T:X_1\times X_2\to Y\) is bounded if
\[
\norm{T}
:=
\sup_{\norm{f}_{X_1}\leq1,\,\norm{g}_{X_2}\leq1}
\norm{T(f,g)}_Y
<\infty.
\]
The resulting operator functional is a norm for \(q\geq1\) and a
quasi-norm for \(q<1\). In either case,
\begin{equation}
\norm{S+T}^{\vartheta}
\leq
\norm{S}^{\vartheta}+\norm{T}^{\vartheta}.
\label{eq:operator-theta-triangle}
\end{equation}
Let \(\mathcal B_2(X_1,X_2;Y)\) denote the vector space of all bounded
bilinear maps from \(X_1\times X_2\) to \(Y\), and define
\begin{equation}
\rho_{\mathcal B}(S,T):=\norm{S-T}^{\vartheta}.
\label{eq:operator-powered-metric}
\end{equation}

\begin{lemma}[Completeness of the bounded bilinear-map space]
\label{lem:bounded-bilinear-completeness}
The metric space
\(
(\mathcal B_2(X_1,X_2;Y),\rho_{\mathcal B})
\)
is complete.
\end{lemma}

\begin{proof}
Let \((T_n)\) be \(\rho_{\mathcal B}\)-Cauchy.  For each
\((f,g)\in X_1\times X_2\),
\[
\rho_Y(T_n(f,g),T_m(f,g))
\leq
\rho_{\mathcal B}(T_n,T_m)
\norm{f}_{X_1}^{\vartheta}\norm{g}_{X_2}^{\vartheta}.
\]
Lemma~\ref{lem:weighted-Lq-completeness} therefore gives a pointwise limit
\[
T(f,g):=\lim_{n\to\infty}T_n(f,g)\quad\text{in }Y.
\]
Continuity of addition and scalar multiplication in \(Y\) shows that \(T\)
is bilinear.  If \(n,m\geq N\) and
\(\rho_{\mathcal B}(T_n,T_m)<\varepsilon\), then, for unit vectors \(f,g\),
letting \(m\to\infty\) gives
\[
\rho_Y(T_n(f,g),T(f,g))\leq\varepsilon.
\]
Taking the supremum over the product unit ball yields
\(\rho_{\mathcal B}(T_n,T)\leq\varepsilon\).  For the same fixed \(n\), the
\(\vartheta\)-triangle inequality gives, on the product unit ball,
\[
\norm{T(f,g)}_Y^{\vartheta}
\leq
\norm{T_n(f,g)}_Y^{\vartheta}
+\rho_Y(T_n(f,g),T(f,g))
\leq
\norm{T_n}^{\vartheta}+\varepsilon.
\]
Hence \(T\) is bounded, and \(T_n\to T\) in \(\rho_{\mathcal B}\).
\end{proof}

We call \(T\) compact if
\(T(B_{X_1}\times B_{X_2})\) has compact closure in \(Y\).  A map is
finite-rank if the linear span of its range is finite-dimensional.  We write
\begin{align*}
\mathcal K_2(X_1,X_2;Y)
&:=\{T\in\mathcal B_2(X_1,X_2;Y):T\text{ is compact}\},\\
\mathcal F_2(X_1,X_2;Y)
&:=\{T\in\mathcal B_2(X_1,X_2;Y):T\text{ is finite-rank}\}.
\end{align*}

The class \(\mathcal K_2(X_1,X_2;Y)\) is a linear subspace. Sums and scalar
multiples of compact bilinear maps are compact.  It is closed in
\(\rho_{\mathcal B}\) by Lemma~\ref{lem:closure-compact-bilinear}.

The essential (quasi-)norm and finite-rank distance of \(T\) are
\begin{align}
\norm{T}_{\mathrm e}
&:=
\inf_{K\in\mathcal K_2(X_1,X_2;Y)}\norm{T-K},
\label{eq:essential-norm-definition}\\
\operatorname{dist}_{\mathrm{fr}}(T)
&:=
\inf_{F\in\mathcal F_2(X_1,X_2;Y)}\norm{T-F}.
\label{eq:finite-rank-distance-definition}
\end{align}
When \(q<1\), the first quantity is customarily interpreted as an essential
quasi-norm. We retain the standard symbol \(\norm{T}_{\mathrm e}\).
The linearity of \(\mathcal K_2\) and
\eqref{eq:operator-theta-triangle} give the quotient inequality
\[
\norm{S+T}_{\mathrm e}^{\vartheta}
\leq
\norm{S}_{\mathrm e}^{\vartheta}
+\norm{T}_{\mathrm e}^{\vartheta},
\qquad S,T\in\mathcal B_2(X_1,X_2;Y).
\]
Thus this is the ordinary quotient-norm inequality for \(q\geq1\) and the
powered quotient quasi-norm inequality for \(0<q<1\).

For a bounded subset \(A\subset Y\), define the Hausdorff measure of
noncompactness with the finite-net quasi-norm-radius normalisation by
\begin{equation}
\chi_Y(A)
:=
\inf\left\{
\varepsilon>0:
A\subset\bigcup_{j=1}^{N}B_Y(y_j,\varepsilon)
\text{ for some }y_1,\ldots,y_N\in Y
\right\},
\label{eq:Hausdorff-mnc-definition}
\end{equation}
where
\[
B_Y(y,\varepsilon):=\{z\in Y:\norm{z-y}_Y<\varepsilon\}.
\]
Since \((Y,\rho_Y)\) is complete, \(\chi_Y(A)=0\) if and only if \(A\)
has compact closure.  If the metric-radius normalisation associated with
\(\rho_Y\) is used instead, its value is exactly
\(\chi_Y(A)^{\vartheta}\).  This records explicitly the normalisation under
which the essential (quasi-)norm identity below is stated.

For \(n\geq1\), the \(n\)-th bilinear approximation number is
\begin{equation}
a_n^{(2)}(T)
:=
\inf\left\{
\norm{T-F}:
F\in\mathcal F_2(X_1,X_2;Y),\
\operatorname{rank}F<n
\right\},
\label{eq:bilinear-approximation-number}
\end{equation}
where \(\operatorname{rank}F\) is the dimension of the linear span of the
range of \(F\).  For Banach-valued multilinear maps, this is the
approximation-number convention of \cite{FernandezMastyloDaSilva2013}, with
rank defined as the dimension of the linear span of the image.  We adopt the
same formula definitionally for \(0<q<1\).  The arguments below use only the
defining rank infima, their monotonicity, and the fact that their infimum over
\(n\) is the finite-rank distance.  These definitions fix all normalisations
used below.

For comparison with arguments carried out directly in the metric, define
\begin{align*}
e_{\vartheta}(T)
&:=\inf_{K\in\mathcal K_2(X_1,X_2;Y)}
\rho_{\mathcal B}(T,K),\\
d_{\mathrm{fr},\vartheta}(T)
&:=\inf_{F\in\mathcal F_2(X_1,X_2;Y)}
\rho_{\mathcal B}(T,F),\\
a_{n,\vartheta}^{(2)}(T)
&:=\inf_{\substack{F\in\mathcal F_2(X_1,X_2;Y)\\
\operatorname{rank}F<n}}
\rho_{\mathcal B}(T,F).
\end{align*}
Also let \(\chi_{\rho_Y}(A)\) be the Hausdorff measure of noncompactness
defined using open \(\rho_Y\)-balls and arbitrary centres in \(Y\).

\begin{lemma}[Power-normalisation identities]
\label{lem:power-normalisation}
For every bounded bilinear map \(T\), every bounded \(A\subset Y\), and every
\(n\geq1\),
\begin{equation}
\begin{aligned}
e_{\vartheta}(T)&=\norm{T}_{\mathrm e}^{\vartheta},&
d_{\mathrm{fr},\vartheta}(T)
&=\operatorname{dist}_{\mathrm{fr}}(T)^{\vartheta},\\
\chi_{\rho_Y}(A)&=\chi_Y(A)^{\vartheta},&
a_{n,\vartheta}^{(2)}(T)&=a_n^{(2)}(T)^{\vartheta}.
\end{aligned}
\label{eq:power-normalisation-identities}
\end{equation}
\end{lemma}

\begin{proof}
Each identity follows because \(r\mapsto r^{\vartheta}\) is continuous and
strictly increasing on \([0,\infty)\), and therefore commutes with each
nonempty infimum occurring here.  The zero map makes the operator infima
nonempty.  Since \(A\) is bounded, one sufficiently large ball centred at
zero makes the finite-cover class nonempty.  The ball identity
\[
B_{\rho_Y}(y,r^{\vartheta})=B_Y(y,r)
\]
gives the third equality with the precise finite-net convention in
\eqref{eq:Hausdorff-mnc-definition}.
\end{proof}

\section{Tail operators and localisation framework}
\label{sec:localisation}

For $a<c<d<b$, define
\[
\mathcal H_{2,c}^{-}(f,g)
:=
\chi_{(a,c)}\mathcal H_2(f,g),
\]
\[
\mathcal H_{2,d}^{+}(f,g)
:=
\chi_{(d,b)}\mathcal H_2(f,g),
\]
and
\[
\mathcal H_2^{c,d}(f,g)
:=
\chi_{(c,d)}\mathcal H_2(f,g).
\]
It is convenient to write
\[
Q_{c,d}h:=\chi_{(a,c)\cup(d,b)}h,
\qquad h\in Y.
\]
Whenever $\mathcal H_2:X_1\times X_2\to Y$ is bounded, put
\[
\alpha(c)
:=
\norm{\mathcal H_{2,c}^{-}}_{X_1\times X_2\to Y},
\qquad
\beta(d)
:=
\norm{\mathcal H_{2,d}^{+}}_{X_1\times X_2\to Y}.
\]
We also define the combined-tail norm
\begin{equation}
\gamma(c,d)
:=
\norm{Q_{c,d}\mathcal H_2}_{X_1\times X_2\to Y}
=
\norm{\mathcal H_2-\mathcal H_2^{c,d}}_{X_1\times X_2\to Y}.
\label{eq:combined-tail-norm}
\end{equation}
The equality is understood in \(Y\). Possible discrepancies at \(c\) and
\(d\) occur on null sets.  The family \(\gamma(c,d)\) decreases as the
middle window expands.  Hence the directed limit
\begin{equation}
\tau(\mathcal H_2)
:=
\lim_{\substack{c\downarrow a\\d\uparrow b}}\gamma(c,d)
=
\inf_{a<c<d<b}\gamma(c,d)
\label{eq:tail-limit-tau}
\end{equation}
exists in \([0,\infty)\).
Indeed, boundedness of \(\mathcal H_2\) gives the explicit finiteness bound
\begin{equation}
0\leq\gamma(c,d)\leq\norm{\mathcal H_2}<\infty
\qquad (a<c<d<b).
\label{eq:finite-window-tail-radius}
\end{equation}
This finiteness is needed when strict finite-net radii are transported from a
middle image to the full image.  At the powered metric scale, put
\begin{equation}
\tau_{\vartheta}(\mathcal H_2)
:=
\inf_{a<c<d<b}\gamma(c,d)^{\vartheta}
=
\tau(\mathcal H_2)^{\vartheta}.
\label{eq:powered-tail-limit}
\end{equation}
The last equality follows from the same monotone-continuity argument as in
Lemma~\ref{lem:power-normalisation}.
Here and below, the displayed operator quasi-norm is the one fixed in
Section~\ref{sec:preliminaries}.  In particular, for \(0<q<1\) every
directed limit and operator convergence statement is taken in the complete
metric induced by its \(q\)-th power.

The endpoint limits are interpreted according to the geometry of \(I\).
For finite endpoints they are one-sided limits from inside \(I\). For
infinite endpoints they are limits at \(\pm\infty\).

\section{Compactness of middle truncations}
\label{sec:middle-compactness}

\begin{lemma}[Local weighted Hardy estimates]
\label{lem:local-Hardy-estimates}
Let $1<p<\infty$, let $v$ be a measurable weight on $I$, positive and finite
almost everywhere, and suppose that
\[
\int_a^x v(t)^{1-p'}\,dt<\infty
\]
for every $x\in I$. Then, for every $f\in L^p(v;I)$,
\begin{equation}
|Hf(x)|
\leq
\norm{f}_{L^p(v;I)}
\left(\int_a^x v(t)^{1-p'}\,dt\right)^{1/p'},
\qquad x\in I.
\label{eq:pointwise-Hardy-estimate}
\end{equation}
Moreover, whenever $a<x<y<b$,
\begin{equation}
|Hf(y)-Hf(x)|
\leq
\norm{f}_{L^p(v;I)}
\left(\int_x^y v(t)^{1-p'}\,dt\right)^{1/p'}.
\label{eq:increment-Hardy-estimate}
\end{equation}
\end{lemma}

\begin{proof}
Let
\[
N_v:=\{t\in I:v(t)=0\text{ or }v(t)=\infty\}.
\]
By hypothesis, $N_v$ is Lebesgue-null. Replacing $v$ by $1$ on $N_v$ does
not change the weighted $L^p$ norm or any integral of $v^{1-p'}$, since all
changes occur on a null set. We may therefore use, almost everywhere, the
factorisation
\[
|f(t)|=|f(t)|v(t)^{1/p}v(t)^{-1/p}.
\]
Let $x\in I$. By weighted H\"older's inequality,
\begin{align*}
|Hf(x)|
&\leq \int_a^x |f(t)|\,dt\\
&=\int_a^x |f(t)|v(t)^{1/p}v(t)^{-1/p}\,dt\\
&\leq
\left(\int_a^x |f(t)|^p v(t)\,dt\right)^{1/p}
\left(\int_a^x v(t)^{-p'/p}\,dt\right)^{1/p'}.
\end{align*}
Since $-p'/p=1-p'$, estimate \eqref{eq:pointwise-Hardy-estimate} follows.

For $a<x<y<b$,
\[
Hf(y)-Hf(x)=\int_x^y f(t)\,dt.
\]
Applying the same weighted H\"older inequality on $(x,y)$ gives
\begin{align*}
|Hf(y)-Hf(x)|
&\leq
\left(\int_x^y |f(t)|^p v(t)\,dt\right)^{1/p}
\left(\int_x^y v(t)^{1-p'}\,dt\right)^{1/p'}\\
&\leq
\norm{f}_{L^p(v;I)}
\left(\int_x^y v(t)^{1-p'}\,dt\right)^{1/p'}.
\end{align*}
This proves \eqref{eq:increment-Hardy-estimate}.
\end{proof}

\begin{lemma}[Compactness of the middle-truncated operator]
\label{lem:middle-truncation-compact}
For every $a<c<d<b$, the operator
\[
\mathcal H_2^{c,d}:X_1\times X_2\longrightarrow Y
\]
is compact.
\end{lemma}

\begin{proof}
Fix $a<c<d<b$. Since $u\in L^1_{\mathrm{loc}}(I)$,
\[
U_{c,d}:=\int_c^d u(x)\,dx<\infty.
\]
For $i\in\{1,2\}$, define
\[
M_i:=V_i(d)^{1/p_i'}.
\]
By Lemma~\ref{lem:local-Hardy-estimates}, if
$\norm{f}_{X_1}\leq1$ and $\norm{g}_{X_2}\leq1$, then
\begin{equation}
|Hf(x)|\leq M_1,
\qquad
|Hg(x)|\leq M_2,
\qquad x\in[c,d].
\label{eq:uniform-middle-bound}
\end{equation}

For $\delta>0$, let
\[
\omega_i(\delta)
:=
\sup\left\{
\left(\int_x^y v_i(t)^{1-p_i'}\,dt\right)^{1/p_i'}:
 c\leq x\leq y\leq d,\ y-x\leq\delta
\right\}.
\]
Because $v_i^{1-p_i'}\in L^1(c,d)$, absolute continuity of the Lebesgue
integral gives
\begin{equation}
\lim_{\delta\downarrow0}\omega_i(\delta)=0,
\qquad i=1,2.
\label{eq:modulus-vanishing}
\end{equation}

Let
\[
\mathcal P:\quad c=x_0<x_1<\cdots<x_m=d
\]
be a finite partition and write
\[
|\mathcal P|:=\max_{0\leq j<m}(x_{j+1}-x_j).
\]
Define
\begin{equation}
\mathcal F_{\mathcal P}(f,g)(x)
:=
\sum_{j=0}^{m-1}Hf(x_j)Hg(x_j)
\chi_{[x_j,x_{j+1})}(x).
\label{eq:finite-rank-map}
\end{equation}
The range of $\mathcal F_{\mathcal P}$ lies in the finite-dimensional space
\[
E_{\mathcal P}
:=
\operatorname{span}\{\chi_{[x_j,x_{j+1})}:0\leq j<m\}.
\]
Each evaluation map $f\mapsto Hf(x_j)$ is bounded by
Lemma~\ref{lem:local-Hardy-estimates}. Hence $\mathcal F_{\mathcal P}$ is a
bounded finite-rank bilinear operator.  It is compact for every \(q>0\) by
Lemma~\ref{lem:finite-dimensional-range}. Local convexity of \(Y\) is not
needed.

The possible discrepancy between the chosen representatives of
$\mathcal H_2^{c,d}$ and $\mathcal F_{\mathcal P}$ at $c$ is supported on a
singleton and is therefore null for the measure $u(x)\,dx$.  It is thus
enough to estimate points $x\in(c,d)\cap[x_j,x_{j+1})$. For unit vectors
$f\in X_1$ and $g\in X_2$,
\begin{align*}
&|Hf(x)Hg(x)-Hf(x_j)Hg(x_j)|\\
&\quad\leq
|Hf(x)|\,|Hg(x)-Hg(x_j)|
+
|Hg(x_j)|\,|Hf(x)-Hf(x_j)|\\
&\quad\leq
M_1\omega_2(|\mathcal P|)+M_2\omega_1(|\mathcal P|),
\end{align*}
where we used \eqref{eq:uniform-middle-bound} and
\eqref{eq:increment-Hardy-estimate}. It follows that
\begin{align}
&\norm{\mathcal H_2^{c,d}-\mathcal F_{\mathcal P}}_{X_1\times X_2\to Y}
\notag\\
&\quad\leq
U_{c,d}^{1/q}
\bigl(M_1\omega_2(|\mathcal P|)+M_2\omega_1(|\mathcal P|)\bigr).
\label{eq:finite-rank-approximation}
\end{align}
By \eqref{eq:modulus-vanishing}, the right-hand side tends to zero as
$|\mathcal P|\to0$. Thus $\mathcal H_2^{c,d}$ is a
\(\rho_{\mathcal B}\)-limit of bounded finite-rank bilinear operators.  The
limit is taken in the complete space from
Lemma~\ref{lem:bounded-bilinear-completeness}.  Lemma
\ref{lem:closure-compact-bilinear} therefore shows that
\(\mathcal H_2^{c,d}\) is compact.
\end{proof}

The same construction gives a quantitative estimate.  For a partition
\(\mathcal P\) as above, set
\begin{equation}
E_{c,d}(\mathcal P)
:=
U_{c,d}^{1/q}
\bigl(M_1\omega_2(|\mathcal P|)+M_2\omega_1(|\mathcal P|)\bigr).
\label{eq:middle-approximation-error}
\end{equation}

\begin{proposition}[Quantitative finite-rank approximation]
\label{prop:quantitative-finite-rank}
Let \(a<c<d<b\), and let \(\mathcal P\) contain \(m\) subintervals.
Then \(\operatorname{rank}\mathcal F_{\mathcal P}\leq m\), and
\begin{equation}
\norm{\mathcal H_2^{c,d}-\mathcal F_{\mathcal P}}
\leq E_{c,d}(\mathcal P).
\label{eq:middle-rank-m-error}
\end{equation}
Consequently, for every \(\varepsilon>0\) there are a partition
\(\mathcal P\) and a bounded finite-rank bilinear map
\(\mathcal F_{\mathcal P}\) such that
\begin{equation}
\norm{\mathcal H_2^{c,d}-\mathcal F_{\mathcal P}}<\varepsilon.
\label{eq:middle-finite-rank-density}
\end{equation}
If, in addition, \(\mathcal H_2:X_1\times X_2\to Y\) is bounded, then for
the full operator
\begin{equation}
a_{m+1}^{(2)}(\mathcal H_2)
\leq
\left(
\gamma(c,d)^q+E_{c,d}(\mathcal P)^q
\right)^{1/q}.
\label{eq:approximation-number-bound}
\end{equation}
In particular, if \(\mathcal P_m\) is the uniform partition of \([c,d]\)
into \(m\) subintervals, then
\begin{align}
a_{m+1}^{(2)}(\mathcal H_2)
&\leq
\Biggl[
\gamma(c,d)^q
+U_{c,d}
\biggl(
M_1\omega_2\!\left(\frac{d-c}{m}\right)
+M_2\omega_1\!\left(\frac{d-c}{m}\right)
\biggr)^q
\Biggr]^{1/q}.
\label{eq:uniform-partition-approximation-bound}
\end{align}
\end{proposition}

\begin{proof}
The range of \(\mathcal F_{\mathcal P}\) is contained in the span of the
\(m\) interval indicators used in \eqref{eq:finite-rank-map}, so its rank is
at most \(m\).  Estimate \eqref{eq:middle-rank-m-error} is
\eqref{eq:finite-rank-approximation}.  Since
\(\omega_i(\delta)\to0\), a sufficiently fine partition gives
\eqref{eq:middle-finite-rank-density}.  Assume now, for the full-operator
assertions, that \(\mathcal H_2\) is bounded.  The two functions
\[
Q_{c,d}\mathcal H_2(f,g)
\quad\text{and}\quad
\mathcal H_2^{c,d}(f,g)-\mathcal F_{\mathcal P}(f,g)
\]
have disjoint supports up to null sets.  Therefore, for unit vectors \(f\)
and \(g\),
\begin{align*}
\norm{\mathcal H_2(f,g)-\mathcal F_{\mathcal P}(f,g)}_Y^q
&=
\norm{Q_{c,d}\mathcal H_2(f,g)}_Y^q\\
&\quad+
\norm{\mathcal H_2^{c,d}(f,g)
-\mathcal F_{\mathcal P}(f,g)}_Y^q\\
&\leq
\gamma(c,d)^q+E_{c,d}(\mathcal P)^q,
\end{align*}
where \eqref{eq:finite-rank-approximation} was used in the last line.
Taking the supremum over the product unit ball and then using
\eqref{eq:bilinear-approximation-number} proves
\eqref{eq:approximation-number-bound}.  The uniform-partition estimate is the
special case \(|\mathcal P_m|=(d-c)/m\).
\end{proof}

\begin{remark}
Under the additional boundedness hypothesis on \(\mathcal H_2\), if
\(v_i^{1-p_i'}\leq L_i\) almost everywhere on \((c,d)\), then
\[
\omega_i(h)\leq L_i^{1/p_i'}h^{1/p_i'}.
\]
Consequently,
\begin{align*}
a_{m+1}^{(2)}(\mathcal H_2)
&\leq
\Biggl[
\gamma(c,d)^q
+U_{c,d}
\biggl(
M_1L_2^{1/p_2'}\left(\frac{d-c}{m}\right)^{1/p_2'}\\
&\hspace{42mm}
+M_2L_1^{1/p_1'}\left(\frac{d-c}{m}\right)^{1/p_1'}
\biggr)^q
\Biggr]^{1/q}.
\end{align*}
For a fixed middle window, the second term therefore decays at least as
\(m^{-\min\{1/p_1',1/p_2'\}}\).  The remaining term \(\gamma(c,d)\)
records the unavoidable contribution from the two endpoints.
\end{remark}

\begin{remark}[Why the estimate survives below \(q=1\)]
\label{rem:quantitative-q-less-than-one}
No Minkowski inequality is used in
\eqref{eq:approximation-number-bound}.  Its two error terms are supported,
respectively, outside and inside \((c,d)\), so
\eqref{eq:disjoint-q-additivity} gives the exact \(q\)-power decomposition
used in the proof of \eqref{eq:approximation-number-bound} for every
\(q>0\).  In particular, for
\(0<q<1\) the right-hand side retains the natural quasi-Banach form
\[
\left(\gamma(c,d)^q+E_{c,d}(\mathcal P)^q\right)^{1/q};
\]
replacing it by an ordinary sum would lose quantitative information.
\end{remark}

\begin{remark}[Relation to linear partition methods]
Endpoint/interior decompositions and partition approximations have linear
antecedents in \cite{EdmundsGurkaPick1994,EdmundsStepanov1994,
StepanovUshakova2024}.  The Hardy-specific refinement used here is the
bilinear sample \(Hf(x_j)Hg(x_j)\), together with a rank bound on its
span-valued range and disjoint-support \(q\)-addition for every \(q>0\).
\end{remark}

\section{The localisation theorem}
\label{sec:main-theorem}

The cutoff argument separates into an abstract compact-middle principle and
a finite-rank enhancement.  This separation isolates the part of the proof
that uses only the geometry of a finite-\(q\) target from the local Hardy
estimates proved in Section~\ref{sec:middle-compactness}.

Let \(E_1,E_2\) be normed spaces, put
\[
D_E:=B_{E_1}\times B_{E_2},
\]
and let \(T:E_1\times E_2\to Y\) be a bounded bilinear map.  In the notation
of Section~\ref{sec:localisation}, define
\[
T^{c,d}:=\chi_{(c,d)}T,
\qquad
\alpha_T(c):=\norm{\chi_{(a,c)}T},
\qquad
\beta_T(d):=\norm{\chi_{(d,b)}T},
\]
and
\begin{equation}
\gamma_T(c,d):=\norm{Q_{c,d}T},
\qquad
\tau(T):=\inf_{a<c<d<b}\gamma_T(c,d).
\label{eq:abstract-tail-limit}
\end{equation}
The cutoffs satisfy \(T=T^{c,d}+Q_{c,d}T\) in \(Y\). The cut points are
null for the measure \(u(x)\,dx\).  The definitions of compactness,
essential distance, Hausdorff measure of noncompactness, finite rank, and
their powered counterparts from Section~\ref{sec:preliminaries} apply
verbatim with \(E_i\) in place of \(X_i\).  For \(T=\mathcal H_2\), the
quantities in \eqref{eq:abstract-tail-limit} are exactly those introduced in
Section~\ref{sec:localisation}.

\subsection{The compact-middle tier}

\begin{theorem}[Abstract compact-middle localisation]
\label{thm:abstract-compact-middle-localisation}
Let \(E_1,E_2\) be normed spaces and let
\(T:E_1\times E_2\to L^q(u;I)\) be a bounded bilinear map, where
\(0<q<\infty\).  Suppose that \(T^{c,d}\) is compact for every
\(a<c<d<b\).  Then
\begin{equation}
\norm{T}_{\mathrm e}
=
\chi_Y\!\left(T(D_E)\right)
=
\tau(T).
\label{eq:abstract-essential-mnc-tail}
\end{equation}
Equivalently, with \(\vartheta=\min\{1,q\}\),
\begin{equation}
e_{\vartheta}(T)
=
\chi_{\rho_Y}\!\left(T(D_E)\right)
=
\tau(T)^{\vartheta}.
\label{eq:abstract-powered-essential-mnc-tail}
\end{equation}
Moreover,
\begin{equation}
T\text{ is compact}
\quad\Longleftrightarrow\quad
\tau(T)=0
\quad\Longleftrightarrow\quad
\lim_{c\downarrow a}\alpha_T(c)
=
\lim_{d\uparrow b}\beta_T(d)=0.
\label{eq:abstract-compactness-equivalence}
\end{equation}
If
\[
\alpha_{T,*}:=\lim_{c\downarrow a}\alpha_T(c),
\qquad
\beta_{T,*}:=\lim_{d\uparrow b}\beta_T(d),
\]
then
\begin{equation}
\max\{\alpha_{T,*},\beta_{T,*}\}
\leq
\tau(T)
\leq
\bigl(\alpha_{T,*}^q+\beta_{T,*}^q\bigr)^{1/q}.
\label{eq:abstract-tail-sandwich}
\end{equation}
\end{theorem}

\begin{proof}
We first record the common cutoff fact.  If
\(K:E_1\times E_2\to Y\) is compact, then
\(\overline{K(D_E)}\) is compact in \(Y\).  Along any endpoint exhaustion,
the tail indicators converge pointwise to zero, so
Lemma~\ref{lem:uniform-absolute-continuity} gives
\begin{equation}
\lim_{\substack{c\downarrow a\\d\uparrow b}}
\norm{Q_{c,d}K}=0.
\label{eq:compact-map-uniform-tail}
\end{equation}
Monotonicity of the tail supports promotes convergence along an exhaustion
to the displayed directed limit.

If \(T\) is compact, \eqref{eq:compact-map-uniform-tail} applied to
\(T\) gives \(\tau(T)=0\).  Conversely, if \(\tau(T)=0\), choose expanding
middle windows \((c_n,d_n)\) with
\(\gamma_T(c_n,d_n)\to0\).  Since each \(T^{c_n,d_n}\) is compact and
\[
\norm{T-T^{c_n,d_n}}=\gamma_T(c_n,d_n),
\]
Lemma~\ref{lem:closure-compact-bilinear} shows that \(T\) is compact.

For the essential distance, compactness of \(T^{c,d}\) gives
\[
\norm{T}_{\mathrm e}
\leq
\norm{T-T^{c,d}}
=
\gamma_T(c,d).
\]
Taking the infimum yields
\begin{equation}
\norm{T}_{\mathrm e}\leq\tau(T).
\label{eq:essential-upper-tail}
\end{equation}
Conversely,
let \(K:E_1\times E_2\to Y\) be compact.  Since \(Q_{c,d}\) is a
contraction and \(\vartheta=\min\{1,q\}\),
\begin{equation}
\gamma_T(c,d)^{\vartheta}
\leq
\norm{T-K}^{\vartheta}
+
\norm{Q_{c,d}K}^{\vartheta}.
\label{eq:compact-perturbation-quasi-tail}
\end{equation}
Use \eqref{eq:compact-map-uniform-tail}, then take the infimum over
compact \(K\), to obtain
\(\tau(T)\leq\norm{T}_{\mathrm e}\).  Hence
\begin{equation}
\norm{T}_{\mathrm e}=\tau(T).
\label{eq:abstract-essential-tail}
\end{equation}

We next compare the Hausdorff measure of noncompactness with the same tail.
Fix \(c,d\).  Since \(T^{c,d}(D_E)\) has compact closure, it has a finite
\(\varepsilon\)-net \(z_1,\ldots,z_N\) for every \(\varepsilon>0\).  The
\(\vartheta\)-triangle inequality shows that these centres form a finite
\[
\bigl(\gamma_T(c,d)^{\vartheta}
+\varepsilon^{\vartheta}\bigr)^{1/\vartheta}
\]
net for \(T(D_E)\).  Letting \(\varepsilon\downarrow0\) and then infimising
over \(c,d\) gives
\begin{equation}
\chi_Y\!\left(T(D_E)\right)\leq\tau(T).
\label{eq:mnc-upper-tail}
\end{equation}

For the reverse inequality, suppose \(y_1,\ldots,y_N\) is a finite
\(r\)-net for \(T(D_E)\).  For each \((x_1,x_2)\in D_E\), choose \(j\)
such that \(\norm{T(x_1,x_2)-y_j}_Y<r\).  Indicator contraction and the
\(\vartheta\)-triangle inequality give
\[
\norm{Q_{c,d}T(x_1,x_2)}_Y^{\vartheta}
\leq
r^{\vartheta}
+
\max_{1\leq j\leq N}\norm{Q_{c,d}y_j}_Y^{\vartheta}.
\]
The maximum tends to zero as the middle interval expands.  Taking the
supremum over \(D_E\), passing to the directed limit, and then infimising
over all finite nets gives
\begin{equation}
\tau(T)\leq\chi_Y\!\left(T(D_E)\right).
\label{eq:abstract-mnc-lower}
\end{equation}
Equations \eqref{eq:abstract-essential-tail},
\eqref{eq:mnc-upper-tail}, and \eqref{eq:abstract-mnc-lower} prove
\eqref{eq:abstract-essential-mnc-tail}.  Lemma~\ref{lem:power-normalisation}
then gives \eqref{eq:abstract-powered-essential-mnc-tail}, with no loss of a
constant when \(0<q<1\).

Finally, disjoint-support \(q\)-additivity gives, for every \(c,d\),
\[
\max\{\alpha_T(c),\beta_T(d)\}
\leq
\gamma_T(c,d)
\leq
\bigl(\alpha_T(c)^q+\beta_T(d)^q\bigr)^{1/q}.
\]
Passing to the endpoint limits proves \eqref{eq:abstract-tail-sandwich} and
the last equivalence in \eqref{eq:abstract-compactness-equivalence}.
\end{proof}

The abstract tier applies immediately to the present operator because the
interior local mechanism has already been established.

\begin{theorem}[Localisation for the weighted bilinear Hardy operator]
\label{thm:central-localisation}
Under the assumptions above, the following statements are equivalent.

\begin{enumerate}[label=\textnormal{(\roman*)},leftmargin=2.2em]
\item
The operator
\[
\mathcal H_2:X_1\times X_2\longrightarrow Y
\]
is compact.

\item
The operator $\mathcal H_2$ is bounded and
\begin{equation}
\lim_{c\downarrow a}\alpha(c)=0,
\label{eq:lower-tail-vanishing}
\end{equation}
\begin{equation}
\lim_{d\uparrow b}\beta(d)=0.
\label{eq:upper-tail-vanishing}
\end{equation}

\item
The operator $\mathcal H_2$ is bounded, and for every $a<c<d<b$ the operator
\[
\mathcal H_2^{c,d}:X_1\times X_2\longrightarrow Y
\]
is compact. Moreover,
\begin{equation}
\lim_{\substack{c\downarrow a\\d\uparrow b}}
\norm{\mathcal H_2-\mathcal H_2^{c,d}}_{X_1\times X_2\to Y}
=0.
\label{eq:middle-truncation-convergence}
\end{equation}
\end{enumerate}
Whenever \(\mathcal H_2\) is bounded, and hence in particular whenever the
equivalent conditions above hold, for every $a<c<d<b$,
\begin{equation}
\norm{\mathcal H_2-\mathcal H_2^{c,d}}_{X_1\times X_2\to Y}
\leq
\bigl(\alpha(c)^q+\beta(d)^q\bigr)^{1/q}.
\label{eq:truncation-error-estimate}
\end{equation}
\end{theorem}

\begin{proof}
Lemma~\ref{lem:middle-truncation-compact} supplies compactness of every
\(\mathcal H_2^{c,d}\).  For a bounded \(\mathcal H_2\), the two endpoint
pieces have disjoint supports, so
\[
\norm{(\mathcal H_2-\mathcal H_2^{c,d})(f,g)}_Y^q
=
\norm{\chi_{(a,c)}\mathcal H_2(f,g)}_Y^q
+
\norm{\chi_{(d,b)}\mathcal H_2(f,g)}_Y^q.
\]
Taking the supremum over the product unit ball proves
\eqref{eq:truncation-error-estimate}.  Theorem
\ref{thm:abstract-compact-middle-localisation}, applied to
\(T=\mathcal H_2\), now shows that compactness is equivalent to vanishing
combined tail, which by \eqref{eq:truncation-error-estimate} and the reverse
restriction bounds is equivalent to
\eqref{eq:lower-tail-vanishing}--\eqref{eq:upper-tail-vanishing}.
The same theorem identifies vanishing combined tail with
\eqref{eq:middle-truncation-convergence}.  This proves all three
equivalences.
\end{proof}

\subsection{The finite-rank enhancement}

The next hypothesis is genuinely additional. Compact bilinear maps are not
being identified in general with the operator-(quasi-)norm closure of
finite-rank bilinear maps.

\begin{theorem}[Abstract finite-rank enhancement]
\label{thm:abstract-finite-rank-enhancement}
Assume the hypotheses of
Theorem~\ref{thm:abstract-compact-middle-localisation}.  Suppose in addition
that
\begin{equation}
T^{c,d}\in
\overline{\mathcal F_2(E_1,E_2;Y)}^{\,\norm{\cdot}}
\qquad (a<c<d<b).
\label{eq:abstract-middle-finite-rank-density}
\end{equation}
Then
\begin{equation}
\operatorname{dist}_{\mathrm{fr}}(T)
=
\tau(T),
\label{eq:abstract-finite-rank-tail}
\end{equation}
and, in the powered metric,
\begin{equation}
d_{\mathrm{fr},\vartheta}(T)=\tau(T)^{\vartheta}.
\label{eq:abstract-powered-finite-rank-tail}
\end{equation}
Moreover,
\begin{equation}
\lim_{n\to\infty}a_n^{(2)}(T)
=
\operatorname{dist}_{\mathrm{fr}}(T)
=
\tau(T),
\label{eq:abstract-approximation-number-limit}
\end{equation}
with the equivalent \(\vartheta\)-powered identity.
\end{theorem}

\begin{proof}
Every finite-rank bilinear map is compact by
Lemma~\ref{lem:finite-dimensional-range}. Hence
\[
\norm{T}_{\mathrm e}
\leq
\operatorname{dist}_{\mathrm{fr}}(T).
\]
Fix \(c,d\), and choose a finite-rank \(F\) approximating \(T^{c,d}\) in
operator quasi-norm.  The \(\vartheta\)-triangle inequality gives
\[
\norm{T-F}^{\vartheta}
\leq
\gamma_T(c,d)^{\vartheta}
+
\norm{T^{c,d}-F}^{\vartheta}.
\]
Letting the approximation error tend to zero and then taking the infimum
over the middle windows yields
\[
\operatorname{dist}_{\mathrm{fr}}(T)\leq\tau(T).
\]
Theorem~\ref{thm:abstract-compact-middle-localisation} gives
\(\norm{T}_{\mathrm e}=\tau(T)\), proving
\eqref{eq:abstract-finite-rank-tail}.
Lemma~\ref{lem:power-normalisation} gives
\eqref{eq:abstract-powered-finite-rank-tail}.

For \(n\geq1\), let
\[
\mathscr A_n
:=
\{F\in\mathcal F_2(E_1,E_2;Y):\operatorname{rank}F<n\}.
\]
The classes \(\mathscr A_n\) increase and their union is the class of all
finite-rank bilinear maps.  Therefore the defining infima decrease to
\(\operatorname{dist}_{\mathrm{fr}}(T)\), which proves
\eqref{eq:abstract-approximation-number-limit}.  The powered statement again
follows from Lemma~\ref{lem:power-normalisation}.
\end{proof}

For \(\mathcal H_2\), the additional hypothesis
\eqref{eq:abstract-middle-finite-rank-density} is supplied by the explicit
piecewise-constant sampler of Proposition~\ref{prop:quantitative-finite-rank}.
Thus the abstract enhancement yields the following concrete headline result.

\begin{theorem}[Exact quantitative identities for the weighted bilinear
Hardy operator]
\label{thm:essential-norm}
Assume that
\(\mathcal H_2:X_1\times X_2\to Y\) is bounded.  Then
\begin{equation}
e_{\vartheta}(\mathcal H_2)
=
d_{\mathrm{fr},\vartheta}(\mathcal H_2)
=
\chi_{\rho_Y}\!\left(
\mathcal H_2(B_{X_1}\times B_{X_2})
\right)
=
\tau_{\vartheta}(\mathcal H_2).
\label{eq:exact-powered-essential-norm}
\end{equation}
Equivalently, in the quasi-norm-radius normalisation,
\begin{equation}
\norm{\mathcal H_2}_{\mathrm e}
=
\operatorname{dist}_{\mathrm{fr}}(\mathcal H_2)
=
\chi_Y\!\left(
\mathcal H_2(B_{X_1}\times B_{X_2})
\right)
=
\tau(\mathcal H_2).
\label{eq:exact-essential-norm}
\end{equation}
If
\[
\alpha_*:=\lim_{c\downarrow a}\alpha(c),
\qquad
\beta_*:=\lim_{d\uparrow b}\beta(d),
\]
then
\begin{equation}
\max\{\alpha_*,\beta_*\}
\leq
\norm{\mathcal H_2}_{\mathrm e}
\leq
\bigl(\alpha_*^q+\beta_*^q\bigr)^{1/q}.
\label{eq:essential-norm-tail-sandwich}
\end{equation}
In particular, \(\mathcal H_2\) is compact if and only if its essential
(quasi-)norm is zero, equivalently if and only if both endpoint-tail limits
vanish.
\end{theorem}

\begin{proof}
Lemma~\ref{lem:middle-truncation-compact} and
Proposition~\ref{prop:quantitative-finite-rank} verify, respectively, the
compact-middle and finite-rank-density hypotheses of
Theorems~\ref{thm:abstract-compact-middle-localisation}
and~\ref{thm:abstract-finite-rank-enhancement}.  Applying those theorems to
\(T=\mathcal H_2\), and using
\(\tau_{\vartheta}(\mathcal H_2)=\tau(\mathcal H_2)^{\vartheta}\), gives
\eqref{eq:exact-powered-essential-norm} and
\eqref{eq:exact-essential-norm}.  The abstract tail sandwich gives
\eqref{eq:essential-norm-tail-sandwich}, and
Theorem~\ref{thm:central-localisation} gives the final compactness
equivalence.
\end{proof}

\begin{corollary}[Limit of the bilinear approximation numbers]
\label{cor:approximation-number-limit}
Under the assumptions of Theorem~\ref{thm:essential-norm},
\begin{equation}
\lim_{n\to\infty}a_n^{(2)}(\mathcal H_2)
=
\tau(\mathcal H_2)
=
\norm{\mathcal H_2}_{\mathrm e}.
\label{eq:approximation-number-limit}
\end{equation}
Equivalently,
\[
\lim_{n\to\infty}a_{n,\vartheta}^{(2)}(\mathcal H_2)
=\tau_{\vartheta}(\mathcal H_2)
=e_{\vartheta}(\mathcal H_2).
\]
Consequently, \(\mathcal H_2\) is compact if and only if
\(a_n^{(2)}(\mathcal H_2)\to0\).
\end{corollary}

\begin{proof}
This is Theorem~\ref{thm:abstract-finite-rank-enhancement} applied to
\(T=\mathcal H_2\), together with
Theorem~\ref{thm:central-localisation}.
\end{proof}

\begin{remark}[Why the abstract theorem remains Hardy-specific in its
application]
The identities above are not consequences of compactness or finite-rank
definitions alone.  The abstract tier requires a cutoff family that
annihilates compact images uniformly and compactness of every middle cutoff.
For \(\mathcal H_2\), Lemma~\ref{lem:middle-truncation-compact} proves the
latter property from the local weighted Hardy estimates.  Equality with
finite-rank distance additionally uses the bilinear sampler
\(\mathcal F_{\mathcal P}\) from
Proposition~\ref{prop:quantitative-finite-rank}.  Thus the abstract theorem
isolates the reusable localization engine, while its hypotheses are realised
here by the operator-specific interior argument.
\end{remark}

\section{The four interval geometries}
\label{sec:geometries}

The notation in Theorem~\ref{thm:central-localisation} covers all four
nonempty open-interval geometries. The endpoint limits are given below.

\begin{center}
\begin{tabular}{@{}lll@{}}
\toprule
Interval \(I\) & lower-tail limit & upper-tail limit\\
\midrule
\((a,b)\) & \(c\downarrow a\) & \(d\uparrow b\)\\
\((a,\infty)\) & \(c\downarrow a\) & \(d\to+\infty\)\\
\((-\infty,b)\) & \(c\to-\infty\) & \(d\uparrow b\)\\
\(\mathbb R\) & \(c\to-\infty\) & \(d\to+\infty\)\\
\bottomrule
\end{tabular}
\end{center}

The endpoint mechanism can be realised by one explicit exhaustion in each
case.  The following statement supplies the concrete tail projections used
in the compactness and finite-centre arguments.

\begin{proposition}[Endpoint exhaustions in all interval geometries]
\label{prop:endpoint-exhaustions}
For \(n\geq1\), choose \((c_n,d_n)\) as follows.
\[
\begin{array}{c|cc}
I & c_n & d_n\\ \hline
(a,b) & a+\dfrac{b-a}{n+2} & b-\dfrac{b-a}{n+2}\\[2mm]
(a,\infty) & a+\dfrac1{n+1} & a+n+1\\[2mm]
(-\infty,b) & b-n-1 & b-\dfrac1{n+1}\\[2mm]
\mathbb R & -n & n
\end{array}
\]
Set
\[
E_n:=(a,c_n)\cup(d_n,b),
\qquad
Q_nh:=\chi_{E_n}h,
\]
with the evident interpretation at infinite endpoints.  Then
\(a<c_n<d_n<b\), \(E_n\downarrow\varnothing\), and
\begin{equation}
\rho_Y(Q_nh,0)\longrightarrow0
\qquad(h\in Y).
\label{eq:geometry-tail-pointwise}
\end{equation}
Every \(Q_n\) is nonexpansive in \(\rho_Y\).  Moreover, for every finite set
\(\{y_1,\ldots,y_N\}\subset Y\),
\begin{equation}
\max_{1\leq j\leq N}\norm{Q_ny_j}_Y\longrightarrow0,
\label{eq:finite-centre-tail-vanishing}
\end{equation}
and for every compact \(K\subset Y\),
\begin{equation}
\sup_{h\in K}\norm{Q_nh}_Y\longrightarrow0.
\label{eq:compact-image-tail-vanishing}
\end{equation}
\end{proposition}

\begin{proof}
The displayed choices give nested endpoint sets with empty intersection, so
\(\chi_{E_n}\to0\) pointwise on \(I\).  Equations
\eqref{eq:geometry-tail-pointwise} and the nonexpansiveness assertion follow
from Lemma~\ref{lem:weighted-Lq-completeness}.  Equation
\eqref{eq:finite-centre-tail-vanishing} follows by taking the maximum of
finitely many convergent sequences.  Finally,
\eqref{eq:compact-image-tail-vanishing} is
Lemma~\ref{lem:uniform-absolute-continuity}.  Thus the pointwise,
finite-centre, and compact-image tail hypotheses used in Section~5 hold for
all four geometries without any additional endpoint convention.
\end{proof}

For the left-infinite and whole-line cases the Hardy primitive is the
genuine improper primitive
\[
H_{-\infty}f(x)=\int_{(-\infty,x]}f(t)\,\dd t.
\]
It is not replaced by a finite lower truncation depending on the chosen
middle window.

\section{Transfer from boundedness characteristics}
\label{sec:characteristics}

The localisation theorem becomes especially useful when a quantitative
boundedness characterisation is already available.  The following observation
makes this transfer precise.

\begin{corollary}[Truncated-characteristic principle]
\label{cor:truncated-characteristic}
Let \(\mathscr W\) be a class of nonnegative measurable output weights.
Suppose that \(u\), \(u\chi_{(a,c)}\), and \(u\chi_{(d,b)}\) belong to
\(\mathscr W\) for every \(a<c<d<b\).  Suppose further that a nonnegative
functional $\mathcal A(w;v_1,v_2)$ and constants $C_1,C_2>0$, depending only
on $p_1,p_2,q$ and independent of \(w\in\mathscr W\), satisfy
\begin{equation}
C_1\mathcal A(w;v_1,v_2)
\leq
\norm{\mathcal H_2}_{L^{p_1}(v_1;I)\times L^{p_2}(v_2;I)\to L^q(w;I)}
\leq
C_2\mathcal A(w;v_1,v_2)
\label{eq:boundedness-characteristic-equivalence}
\end{equation}
for every $w\in\mathscr W$. Then
\[
\mathcal H_2:
L^{p_1}(v_1;I)\times L^{p_2}(v_2;I)
\longrightarrow L^q(u;I)
\]
is compact if and only if
\[
\mathcal A(u;v_1,v_2)<\infty,
\]
\[
\lim_{c\downarrow a}
\mathcal A(u\chi_{(a,c)};v_1,v_2)=0,
\]
and
\[
\lim_{d\uparrow b}
\mathcal A(u\chi_{(d,b)};v_1,v_2)=0.
\]
Moreover, if
\begin{align*}
\mathcal A_-^*
&:=
\limsup_{c\downarrow a}
\mathcal A(u\chi_{(a,c)};v_1,v_2),\\
\mathcal A_+^*
&:=
\limsup_{d\uparrow b}
\mathcal A(u\chi_{(d,b)};v_1,v_2),
\end{align*}
then
\begin{equation}
C_1\max\{\mathcal A_-^*,\mathcal A_+^*\}
\leq
\norm{\mathcal H_2}_{\mathrm e}
\leq
C_2\bigl((\mathcal A_-^*)^q+(\mathcal A_+^*)^q\bigr)^{1/q}.
\label{eq:characteristic-essential-norm}
\end{equation}
\end{corollary}

\begin{proof}
By the assumed membership and uniformity, apply
\eqref{eq:boundedness-characteristic-equivalence} to the output weights
\[
u,
\qquad
u\chi_{(a,c)},
\qquad
u\chi_{(d,b)}.
\]
Then
\[
\alpha(c)\asymp\mathcal A(u\chi_{(a,c)};v_1,v_2),
\qquad
\beta(d)\asymp\mathcal A(u\chi_{(d,b)};v_1,v_2),
\]
with constants independent of $c$ and $d$. The conclusion follows from
Theorems~\ref{thm:central-localisation} and~\ref{thm:essential-norm}.
\end{proof}

All boundedness theorems used below are commonly stated for nonnegative
inputs.  They determine the same norm for real- or complex-valued inputs,
because \(|Hf|\leq H|f|\) and taking absolute values preserves the weighted
input norms.

For the instantiations below, \(\mathscr W=M_+(I)\), the class of
nonnegative measurable output weights on the general interval \((a,b)\),
with infinite endpoints allowed.  Theorem~A of
\cite{KanjilalPerssonShambilova2019} is formulated on this class in all five
finite-exponent regimes, and its comparison constants depend only on
\(p_1,p_2,q\).  The class is stable under multiplication by interval
indicators.  Consequently, \(u\), \(u\chi_{(a,c)}\), and
\(u\chi_{(d,b)}\) meet the membership requirement above, with constants
independent of \(c,d\).

\subsection{An explicit compactness criterion in the convex range}

We now instantiate Corollary~\ref{cor:truncated-characteristic} with a
classical boundedness characteristic.  Assume in this subsection that
\begin{equation}
1<p_1,p_2<\infty,
\qquad
\max(p_1,p_2)\le q<\infty.
\label{eq:convex-bilinear-range}
\end{equation}
Put
\[
V_i(x):=\int_a^x v_i(t)^{1-p_i'}\,\dd t,
\qquad i=1,2,
\]
and, for any nonnegative measurable output weight $w$, define
\begin{equation}
\mathcal D(w)
:=
\sup_{a<x<b}
\left(\int_x^b w(t)\,\dd t\right)^{1/q}
V_1(x)^{1/p_1'}V_2(x)^{1/p_2'}.
\label{eq:D-characteristic}
\end{equation}
Theorem~1 of Aguilar Ca\~nestro, Ortega Salvador and Ram\'irez Torreblanca
\cite{AguilarCanestroOrtegaRamirez2012} gives the classical formula for
positive weights.  The general-interval \(M_+\) formulation needed for
truncated output weights is Theorem~A(i) of
\cite{KanjilalPerssonShambilova2019}. See also Theorem~B and Theorem~3 of
\cite{GogatishviliJainKanjilal2022}.  These results give
\begin{equation}
\norm{\mathcal H_2}_{L^{p_1}(v_1;I)\times L^{p_2}(v_2;I)\to L^q(w;I)}
\asymp_{p_1,p_2,q}
\mathcal D(w).
\label{eq:D-norm-equivalence}
\end{equation}

For $a<c<d<b$, set
\begin{align}
\mathcal D_-(c)
&:=\mathcal D(u\chi_{(a,c)}) \\
&=
\sup_{a<x<c}
\left(\int_x^c u(t)\,\dd t\right)^{1/q}
V_1(x)^{1/p_1'}V_2(x)^{1/p_2'},
\label{eq:D-lower}
\end{align}
and
\begin{align}
\mathcal D_+(d)
&:=\mathcal D(u\chi_{(d,b)}) \\
&=
\sup_{a<x<b}
\left(\int_{\max\{x,d\}}^b u(t)\,\dd t\right)^{1/q}
V_1(x)^{1/p_1'}V_2(x)^{1/p_2'}.
\label{eq:D-upper}
\end{align}
The second formula is simply the characteristic
\eqref{eq:D-characteristic} evaluated at the truncated output weight. It is
written in this form to avoid imposing any additional continuity assumption
on the functions $V_i$ at $d$.

\begin{corollary}[Explicit Muckenhoupt-type compactness criterion]
\label{cor:explicit-convex-compactness}
Assume the hypotheses of Theorem~\ref{thm:central-localisation} together with
\eqref{eq:convex-bilinear-range}.  Then
\[
\mathcal H_2:
L^{p_1}(v_1;I)\times L^{p_2}(v_2;I)
\longrightarrow L^q(u;I)
\]
is compact if and only if
\begin{equation}
\mathcal D(u)<\infty,
\label{eq:D-global-finite}
\end{equation}
\begin{equation}
\lim_{c\downarrow a}\mathcal D_-(c)=0,
\label{eq:D-lower-vanish}
\end{equation}
and
\begin{equation}
\lim_{d\uparrow b}\mathcal D_+(d)=0.
\label{eq:D-upper-vanish}
\end{equation}
The endpoint limits are interpreted according to the geometry of $I$ as in
Section~\ref{sec:geometries}.
\end{corollary}

\begin{proof}
Equation~\eqref{eq:D-norm-equivalence} verifies the hypothesis of
Corollary~\ref{cor:truncated-characteristic} with
$\mathcal A=\mathcal D$.  Applying that corollary to $u$,
$u\chi_{(a,c)}$, and $u\chi_{(d,b)}$ gives exactly
\eqref{eq:D-global-finite}--\eqref{eq:D-upper-vanish}.
\end{proof}

\begin{remark}
Corollary~\ref{cor:explicit-convex-compactness} is restricted to
the range in which the simple characteristic \eqref{eq:D-characteristic} is
known to be equivalent to the bilinear operator norm.  The remaining
exponent regimes require different characteristics.  They are stated next.
\end{remark}

\subsection{Mixed exponent regimes}
\label{subsec:mixed-characteristics}

For a nonnegative measurable output weight \(w\), write
\begin{equation}
W_w(x):=\int_x^b w(t)\,\dd t,
\qquad
\dd V_i(x):=v_i(x)^{1-p_i'}\,\dd x.
\label{eq:tail-weight-and-dVi}
\end{equation}
For any \(s>0\), we use the algebraic companion notation
\begin{equation}
\frac1{s'}:=1-\frac1s.
\label{eq:algebraic-companion}
\end{equation}
Thus \(s'=\infty\) at \(s=1\), whereas \(s'<0\) when \(0<s<1\).
Expressions such as \(r/s'\) always mean \(r(1-1/s)\).  In particular,
\(V_i^{r/q'}=1\) when \(q=1\).  Whenever \(r_i\) is defined below by
\(1/r_i=1/q-1/p_i\), the source exponent is equivalently
\[
\frac{r_i}{p_i'}-1
=r_i\left(1-\frac1{p_i}\right)-1
=r_i\left(1-\frac1q\right)
=\frac{r_i}{q'}.
\]
All characteristics below are extended-valued in \([0,\infty]\). For
\(\lambda>0\), we use \(0^\lambda=0\) and
\(\infty^\lambda=\infty\).

The literal \(\dd V_i\)-formulas below are Theorem~A(ii)--(iii) of
\cite{KanjilalPerssonShambilova2019}, with the supporting Theorems~3--4
there.  Theorems~2.2--2.3 of \cite{Krepela2017} provide the \(q>1\)
iteration antecedent in an alternative normalization.

Assume first that
\begin{equation}
1<p_1\leq q<p_2<\infty,
\qquad
\frac1{r_2}:=\frac1q-\frac1{p_2}.
\label{eq:mixed-range-12}
\end{equation}
Define
\begin{equation}
\mathcal A_2(w)
:=
\sup_{a<x<b}
V_1(x)^{1/p_1'}
\left(
\int_x^b
W_w(y)^{r_2/q}
V_2(y)^{r_2/q'}
\,\dd V_2(y)
\right)^{1/r_2}.
\label{eq:A2-characteristic}
\end{equation}
Then
\begin{equation}
\norm{\mathcal H_2}_{L^{p_1}(v_1)\times L^{p_2}(v_2)\to L^q(w)}
\asymp_{p_1,p_2,q}
\mathcal A_2(w).
\label{eq:A2-norm-equivalence}
\end{equation}

In the symmetric mixed range
\begin{equation}
1<p_2\leq q<p_1<\infty,
\qquad
\frac1{r_1}:=\frac1q-\frac1{p_1},
\label{eq:mixed-range-21}
\end{equation}
put
\begin{equation}
\mathcal A_3(w)
:=
\sup_{a<x<b}
V_2(x)^{1/p_2'}
\left(
\int_x^b
W_w(y)^{r_1/q}
V_1(y)^{r_1/q'}
\,\dd V_1(y)
\right)^{1/r_1}.
\label{eq:A3-characteristic}
\end{equation}
Then
\begin{equation}
\norm{\mathcal H_2}_{L^{p_1}(v_1)\times L^{p_2}(v_2)\to L^q(w)}
\asymp_{p_1,p_2,q}
\mathcal A_3(w).
\label{eq:A3-norm-equivalence}
\end{equation}

For later use, set
\begin{equation}
u_c^-:=u\chi_{(a,c)},
\qquad
u_d^+:=u\chi_{(d,b)}.
\label{eq:truncated-output-weights}
\end{equation}

\begin{corollary}[Compactness in the mixed regimes]
\label{cor:mixed-compactness}
Under the hypotheses of Theorem~\ref{thm:central-localisation}, the following
statements hold.
\begin{enumerate}[label=\textnormal{(\alph*)},leftmargin=2.2em]
\item
If \eqref{eq:mixed-range-12} holds, then \(\mathcal H_2\) is compact if and
only if
\[
\mathcal A_2(u)<\infty,
\qquad
\lim_{c\downarrow a}\mathcal A_2(u_c^-)=0,
\qquad
\lim_{d\uparrow b}\mathcal A_2(u_d^+)=0.
\]
\item
If \eqref{eq:mixed-range-21} holds, then \(\mathcal H_2\) is compact if and
only if
\[
\mathcal A_3(u)<\infty,
\qquad
\lim_{c\downarrow a}\mathcal A_3(u_c^-)=0,
\qquad
\lim_{d\uparrow b}\mathcal A_3(u_d^+)=0.
\]
\end{enumerate}
In each case, the essential (quasi-)norm satisfies the corresponding two-sided
estimate \eqref{eq:characteristic-essential-norm}.
\end{corollary}

\begin{proof}
Apply Corollary~\ref{cor:truncated-characteristic} to
\eqref{eq:A2-norm-equivalence} or \eqref{eq:A3-norm-equivalence}.
\end{proof}

\subsection{Fully subcritical exponent regimes}
\label{subsec:subcritical-characteristics}

Assume in this subsection that
\begin{equation}
0<q<\min(p_1,p_2),
\qquad
\frac1{r_i}:=\frac1q-\frac1{p_i},
\quad i=1,2.
\label{eq:fully-subcritical-range}
\end{equation}
K\v{r}epela's corresponding iteration theorem
\cite{Krepela2017} assumes \(q>1\).  Theorem~A(iv)--(v), Theorems~5--6, and
Remarks~3--4 of \cite{KanjilalPerssonShambilova2019} give the displayed
characteristics for the full range \(0<q<\min(p_1,p_2)\), including
\(q\leq1\).  Their strict boundary in the deep case is the one used below.
The formulas remain literal in the quasi-Banach range.  By
\eqref{eq:Vi-strictly-positive}, every negative power of \(V_i\) is
well-defined at each interior point.
Define
\begin{align}
\mathcal A_4(w)
&:=
\sup_{a<x<b}
V_1(x)^{1/p_1'}
\left(
\int_x^b
W_w(y)^{r_2/q}
V_2(y)^{r_2/q'}
\,\dd V_2(y)
\right)^{1/r_2},
\label{eq:A4-characteristic}\\
\mathcal A_5(w)
&:=
\sup_{a<x<b}
V_2(x)^{1/p_2'}
\left(
\int_x^b
W_w(y)^{r_1/q}
V_1(y)^{r_1/q'}
\,\dd V_1(y)
\right)^{1/r_1}.
\label{eq:A5-characteristic}
\end{align}

If
\begin{equation}
\frac1q\leq\frac1{p_1}+\frac1{p_2},
\label{eq:moderately-subcritical-condition}
\end{equation}
set
\begin{equation}
\mathcal A_{\mathrm{sub},0}(w)
:=
\mathcal A_4(w)+\mathcal A_5(w).
\label{eq:subcritical-A0}
\end{equation}
The published boundedness theorem gives
\begin{equation}
\norm{\mathcal H_2}_{L^{p_1}(v_1)\times L^{p_2}(v_2)\to L^q(w)}
\asymp_{p_1,p_2,q}
\mathcal A_{\mathrm{sub},0}(w).
\label{eq:subcritical-A0-norm}
\end{equation}

Suppose instead that
\begin{equation}
\frac1q>\frac1{p_1}+\frac1{p_2},
\qquad
\frac1\kappa
:=
\frac1q-\frac1{p_1}-\frac1{p_2}.
\label{eq:deep-subcritical-condition}
\end{equation}
Use the algebraic companions
\begin{equation}
\frac1{r_i'}:=1-\frac1{r_i},
\qquad i=1,2.
\label{eq:ri-algebraic-companion}
\end{equation}
Depending on \(q\), \(r_i\) may be below, equal to, or above \(1\), and
\(r_i'\) may therefore be negative, infinite, or positive.  These are
algebraic powers in the published characteristics, not dual exponents of
quasi-Banach spaces.  When \(r_i'=\infty\), the corresponding exponent
\(\kappa/r_i'\) is zero and the factor is interpreted as one.  The outer
exponents in the source formulas satisfy
\[
\frac{\kappa}{p_1'}-1=\frac{\kappa}{r_2'},
\qquad
\frac{\kappa}{p_2'}-1=\frac{\kappa}{r_1'},
\]
because
\(1/r_2=1/q-1/p_2\), \(1/r_1=1/q-1/p_1\), and
\(1/\kappa=1/q-1/p_1-1/p_2\).  Define
\begin{align}
\mathcal A_6(w)
&:=
\left[
\int_a^b
\left(
\int_x^b
W_w(y)^{r_2/q}
V_2(y)^{r_2/q'}
\,\dd V_2(y)
\right)^{\kappa/r_2}
V_1(x)^{\kappa/r_2'}
\,\dd V_1(x)
\right]^{1/\kappa},
\label{eq:A6-characteristic}\\
\mathcal A_7(w)
&:=
\left[
\int_a^b
\left(
\int_x^b
W_w(y)^{r_1/q}
V_1(y)^{r_1/q'}
\,\dd V_1(y)
\right)^{\kappa/r_1}
V_2(x)^{\kappa/r_1'}
\,\dd V_2(x)
\right]^{1/\kappa}.
\label{eq:A7-characteristic}
\end{align}
Set
\begin{equation}
\mathcal A_{\mathrm{sub},1}(w)
:=
\mathcal A_6(w)+\mathcal A_7(w).
\label{eq:subcritical-A1}
\end{equation}
In this regime,
\begin{equation}
\norm{\mathcal H_2}_{L^{p_1}(v_1)\times L^{p_2}(v_2)\to L^q(w)}
\asymp_{p_1,p_2,q}
\mathcal A_{\mathrm{sub},1}(w).
\label{eq:subcritical-A1-norm}
\end{equation}

\begin{corollary}[Compactness in the fully subcritical regimes]
\label{cor:subcritical-compactness}
Assume the hypotheses of Theorem~\ref{thm:central-localisation} and
\eqref{eq:fully-subcritical-range}.
\begin{enumerate}[label=\textnormal{(\alph*)},leftmargin=2.2em]
\item
Under \eqref{eq:moderately-subcritical-condition}, \(\mathcal H_2\) is
compact if and only if
\[
\mathcal A_{\mathrm{sub},0}(u)<\infty,
\qquad
\lim_{c\downarrow a}\mathcal A_{\mathrm{sub},0}(u_c^-)=0,
\qquad
\lim_{d\uparrow b}\mathcal A_{\mathrm{sub},0}(u_d^+)=0.
\]
\item
Under \eqref{eq:deep-subcritical-condition}, \(\mathcal H_2\) is compact if
and only if
\[
\mathcal A_{\mathrm{sub},1}(u)<\infty,
\qquad
\lim_{c\downarrow a}\mathcal A_{\mathrm{sub},1}(u_c^-)=0,
\qquad
\lim_{d\uparrow b}\mathcal A_{\mathrm{sub},1}(u_d^+)=0.
\]
\end{enumerate}
In each case, the essential (quasi-)norm is controlled from both sides by the
corresponding endpoint limits through
\eqref{eq:characteristic-essential-norm}.
\end{corollary}

\begin{proof}
Equations \eqref{eq:subcritical-A0-norm} and
\eqref{eq:subcritical-A1-norm} verify the hypothesis of
Corollary~\ref{cor:truncated-characteristic}.
\end{proof}

\begin{remark}
Together with Corollary~\ref{cor:explicit-convex-compactness} and
Corollary~\ref{cor:mixed-compactness}, the two alternatives in
Corollary~\ref{cor:subcritical-compactness} cover every relative ordering of
\(q,p_1,p_2\) in the range \(1<p_i<\infty\), \(0<q<\infty\).
The split at \eqref{eq:moderately-subcritical-condition} is intrinsic to the
underlying boundedness theorem and cannot be removed by using the mixed
characteristics alone.
\end{remark}

\subsection{Worked examples and compactness thresholds}
\label{subsec:worked-examples}

We illustrate Corollary~\ref{cor:explicit-convex-compactness} by two elementary examples. The first shows
that, on an unbounded interval, the endpoint condition can distinguish
compactness from mere boundedness at a sharp decay threshold. The second
gives a regular bounded-interval model in which both endpoint
characteristics can be evaluated explicitly.

In both examples
\[
p_1=p_2=q=2,
\]
so that \(q=\max(p_1,p_2)\) and the exponent condition of
Corollary~\ref{cor:explicit-convex-compactness} is satisfied. The local hypotheses of Theorem~\ref{thm:central-localisation} are also immediate. The output weight belongs to \(L^1_{\mathrm{loc}}(I)\), and
for the unit input weights one has
\[
V_1(x)=V_2(x)=x<\infty
\]
at every interior point of the corresponding interval.

\medskip
\begin{example}[A sharp decay threshold on the half-line]
\label{ex:halfline-threshold}
Let
\[
I=(0,\infty),\qquad
p_1=p_2=q=2,\qquad
v_1=v_2\equiv1,
\]
and, for \(r>1\), define
\[
u_r(x)=(r-1)(1+x)^{-r}.
\]
Since \(p_1'=p_2'=2\),
\[
V_1(x)=V_2(x)=x.
\]
Moreover,
\[
\int_x^\infty u_r(t)\,dt=(1+x)^{1-r}.
\]
Consequently the characteristic in \eqref{eq:D-characteristic} becomes
\[
\mathcal D(u_r)
=
\sup_{x>0}x(1+x)^{(1-r)/2}.
\]
As \(x\to\infty\),
\[
x(1+x)^{(1-r)/2}
\sim
x^{(3-r)/2}.
\]
On the other hand,
\[
x(1+x)^{(1-r)/2}\longrightarrow0
\qquad (x\downarrow0),
\]
and the function \(x\mapsto x(1+x)^{(1-r)/2}\) is continuous on
\((0,\infty)\). Hence possible divergence of the supremum can occur only
as \(x\to\infty\). It follows that
\[
\mathcal D(u_r)<\infty
\quad\Longleftrightarrow\quad
r\ge3.
\]

For the lower endpoint,
\[
\mathcal D_{r, -}(c)
:=
\mathcal D\!\left(u_r\chi_{(0,c)}\right)
=
\sup_{0<x<c}
x\left(\int_x^c u_r(t)\,dt\right)^{1/2}.
\]
Since
\[
\mathcal D_{r, -}(c)
\le
c\left(\int_0^c u_r(t)\,dt\right)^{1/2}
\le c,
\]
we have
\[
\lim_{c\downarrow0}\mathcal D_{r, -}(c)=0.
\]

For the upper endpoint,
\[
\mathcal D_{r, +}(d)
:=
\mathcal D\!\left(u_r\chi_{(d,\infty)}\right)
=
\sup_{x>0}
x(1+\max\{x,d\})^{(1-r)/2}.
\]
Splitting the supremum at \(x=d\) gives
\[
\mathcal D_{r, +}(d)
=
\max\left\{
d(1+d)^{(1-r)/2},
\,
\sup_{x\ge d}x(1+x)^{(1-r)/2}
\right\},
\]
where the first term is the supremum corresponding to \(0<x<d\).

If \(r>3\), then
\[
d(1+d)^{(1-r)/2}\longrightarrow0
\qquad(d\to\infty),
\]
and
\[
\sup_{x\ge d}x(1+x)^{(1-r)/2}\longrightarrow0
\qquad(d\to\infty).
\]
Hence
\[
\lim_{d\to\infty}\mathcal D_{r, +}(d)=0,
\]
and Corollary~\ref{cor:explicit-convex-compactness} shows that \(\mathcal H_2\) is compact.

At the critical value \(r=3\),
\[
\mathcal D(u_3)
=
\sup_{x>0}\frac{x}{1+x}
=
1.
\]
For every \(d>0\),
\[
\mathcal D_{3, +}(d)
\ge
\sup_{x\ge d}\frac{x}{1+x}
=
1.
\]
Since the defining expression never exceeds \(1\), in fact
\[
\mathcal D_{3, +}(d)=1
\qquad\text{for every }d>0.
\]
Thus the operator is bounded but not compact.

Consequently,
\[
\begin{array}{ccl}
1<r<3
&:&
\mathcal H_2\text{ is unbounded},\\[1mm]
r=3
&:&
\mathcal H_2\text{ is bounded but not compact},\\[1mm]
r>3
&:&
\mathcal H_2\text{ is compact}.
\end{array}
\]
\end{example}

\medskip
\begin{example}[A regular bounded-interval case]
\label{ex:bounded-unit-weight}

Let
\[
I=(0,1),\qquad
p_1=p_2=q=2,\qquad
u=v_1=v_2\equiv1.
\]
Again,
\[
V_1(x)=V_2(x)=x,
\]
and therefore
\[
\mathcal D(u)
=
\sup_{0<x<1}x\sqrt{1-x}.
\]
Since
\[
\frac{d}{dx}\bigl(x\sqrt{1-x}\bigr)
=
\frac{2-3x}{2\sqrt{1-x}},
\]
the maximum occurs at \(x=2/3\). Hence
\[
\mathcal D(u)
=
\frac{2}{3\sqrt3}
<\infty.
\]

For \(0<c<1\),
\[
\mathcal D_{-}(c)
=
\sup_{0<x<c}x\sqrt{c-x}.
\]
Writing \(x=cy\), with \(0<y<1\), gives
\[
x\sqrt{c-x}
=
c^{3/2}y\sqrt{1-y}.
\]
The function \(y\sqrt{1-y}\) attains its maximum at \(y=2/3\), with
maximum value \(2/(3\sqrt3)\). Therefore
\[
\mathcal D_{-}(c)
=
\frac{2}{3\sqrt3}\,c^{3/2},
\]
and hence
\[
\lim_{c\downarrow0}\mathcal D_{-}(c)=0.
\]

For the upper endpoint,
\[
\mathcal D_{+}(d)
=
\sup_{0<x<1}
x\sqrt{1-\max\{x,d\}}.
\]
Splitting the supremum at \(x=d\), we obtain
\[
\mathcal D_{+}(d)
=
\max\left\{
\sup_{0<x<d}x\sqrt{1-d},
\,
\sup_{d\le x<1}x\sqrt{1-x}
\right\}.
\]
The first supremum is
\[
\sup_{0<x<d}x\sqrt{1-d}
=
d\sqrt{1-d}.
\]
If \(d\ge2/3\), the function
\[
x\longmapsto x\sqrt{1-x}
\]
is decreasing on \([d,1)\), and therefore
\[
\sup_{d\le x<1}x\sqrt{1-x}
=
d\sqrt{1-d}.
\]
Consequently,
\[
\mathcal D_{+}(d)
=
d\sqrt{1-d},
\qquad
\frac23\le d<1.
\]
It follows that
\[
\lim_{d\uparrow1}\mathcal D_{+}(d)=0.
\]
Thus, by Corollary~\ref{cor:explicit-convex-compactness},
\[
\mathcal H_2:
L^2(0,1)\times L^2(0,1)
\longrightarrow
L^2(0,1)
\]
is compact.
\end{example}

\medskip
\begin{example}[An asymmetric mixed-range threshold]
\label{ex:asymmetric-mixed-threshold}
Let
\[
I=(0,\infty),
\qquad
p_1=q=2,
\qquad
p_2=4,
\qquad
v_1=v_2\equiv1,
\]
and, for \(r>1\), let
\[
u_r(x):=(r-1)(1+x)^{-r}.
\]
This is the asymmetric mixed regime outside the convex exponent range
\[
1<p_1\leq q<p_2<\infty.
\]
Here \(r_2=4\), \(q'=2\), \(V_1(x)=V_2(x)=x\), and
\[
W_{u_r}(x)=\int_x^\infty u_r(t)\,\dd t=(1+x)^{1-r}.
\]
The characteristic \eqref{eq:A2-characteristic} therefore becomes
\begin{equation}
\mathcal A_2(u_r)
=
\sup_{x>0}
x^{1/2}
\left(
\int_x^\infty
y^2(1+y)^{2(1-r)}\,\dd y
\right)^{1/4}.
\label{eq:asymmetric-A2-global}
\end{equation}

The integral in \eqref{eq:asymmetric-A2-global} is finite for one, and hence
for every, \(x>0\) precisely when \(r>5/2\).  In that range,
\begin{equation}
\int_x^\infty y^2(1+y)^{2(1-r)}\,\dd y
\asymp_r x^{5-2r},
\qquad x\geq1.
\label{eq:asymmetric-integral-asymptotic}
\end{equation}
It follows that the expression under the supremum is comparable, for large
\(x\), to
\[
x^{1/2}x^{(5-2r)/4}=x^{(7-2r)/4}.
\]
Near zero, the integral is finite and multiplication by \(x^{1/2}\) forces
the expression to zero.  Consequently,
\begin{equation}
\mathcal A_2(u_r)<\infty
\quad\Longleftrightarrow\quad
r\geq\frac72.
\label{eq:asymmetric-boundedness-threshold}
\end{equation}

For the lower truncation \(u_{r,c}^-:=u_r\chi_{(0,c)}\), with \(0<c\leq1\),
\[
W_{u_{r,c}^-}(y)
=
\begin{cases}
\displaystyle\int_y^c u_r(t)\,\dd t,&0<y<c,\\
0,&y\geq c.
\end{cases}
\]
Since \(W_{u_{r,c}^-}(y)\leq (r-1)c\) on \((0,c)\),
\begin{align*}
\mathcal A_2(u_{r,c}^-)
&\leq
\sup_{0<x<c}
x^{1/2}
\left(
(r-1)^2c^2\int_x^c y^2\,\dd y
\right)^{1/4}\\
&\leq C_r c^{7/4}.
\end{align*}
Thus
\begin{equation}
\lim_{c\downarrow0}\mathcal A_2(u_{r,c}^-)=0.
\label{eq:asymmetric-lower-tail}
\end{equation}

For the upper truncation \(u_{r,d}^+:=u_r\chi_{(d,\infty)}\),
\[
W_{u_{r,d}^+}(y)
=(1+\max\{y,d\})^{1-r}.
\]
If \(r\geq7/2\) and \(d\geq1\), splitting the defining integral at \(d\) and
using \eqref{eq:asymmetric-integral-asymptotic} gives
\begin{equation}
\mathcal A_2(u_{r,d}^+)
\asymp_r d^{(7-2r)/4}.
\label{eq:asymmetric-upper-tail-rate}
\end{equation}
Indeed, the lower estimate follows by choosing \(x=d\).  For \(x<d\), the
integral is at most a constant multiple of \(d^{5-2r}\), while for
\(x\geq d\) estimate \eqref{eq:asymmetric-integral-asymptotic} applies
directly.

At \(r=7/2\), the upper-tail characteristic remains bounded away from zero.
More precisely,
\[
\lim_{x\to\infty}
x^{1/2}
\left(
\int_x^\infty y^2(1+y)^{-5}\,\dd y
\right)^{1/4}
=2^{-1/4}.
\]
For \(r>7/2\), \eqref{eq:asymmetric-upper-tail-rate} tends to zero.  Hence
Corollary~\ref{cor:mixed-compactness} yields the sharp classification
\[
\begin{array}{ccl}
1<r<7/2
&:&
\mathcal H_2\text{ is unbounded},\\[1mm]
r=7/2
&:&
\mathcal H_2\text{ is bounded but not compact},\\[1mm]
r>7/2
&:&
\mathcal H_2\text{ is compact}.
\end{array}
\]
This example is asymmetric in the two input exponents and lies outside the
convex range of Corollary~\ref{cor:explicit-convex-compactness}.
\end{example}

\medskip
\begin{example}[A genuinely quasi-Banach target with an explicit
approximation upper bound]
\label{ex:quasi-banach-unit-weight}
Let
\[
I=(0,1),\qquad
p_1=p_2=2,\qquad
q=\frac12,\qquad
u=v_1=v_2\equiv1.
\]
The target \(L^{1/2}(0,1)\) is nonlocally convex, and its quasi-norm is
\[
\norm{h}_{1/2}
=
\left(\int_0^1 |h(x)|^{1/2}\,\dd x\right)^2.
\]
For \(\norm{f}_2,\norm{g}_2\leq1\), the Cauchy--Schwarz inequality gives
\[
|Hf(x)|\leq x^{1/2},
\qquad
|Hg(x)|\leq x^{1/2},
\qquad 0<x<1.
\]
Consequently,
\begin{equation}
|\mathcal H_2(f,g)(x)|^{1/2}\leq x^{1/2}.
\label{eq:qhalf-pointwise-majorant}
\end{equation}
In particular,
\[
\norm{\mathcal H_2(f,g)}_{1/2}
\leq
\left(\int_0^1x^{1/2}\,\dd x\right)^2
=\frac49,
\]
so \(\mathcal H_2\) is bounded.  The same majorant gives the explicit tail
estimates
\begin{align}
\alpha(c)
&\leq
\left(\int_0^c x^{1/2}\,\dd x\right)^2
=\frac49c^3,
\label{eq:qhalf-lower-tail}\\
\beta(d)
&\leq
\left(\int_d^1 x^{1/2}\,\dd x\right)^2
=\frac49\bigl(1-d^{3/2}\bigr)^2.
\label{eq:qhalf-upper-tail}
\end{align}
Both limits vanish, and Theorem~\ref{thm:central-localisation} therefore
proves compactness directly in \(L^{1/2}(0,1)\).

This example lies in the deep fully subcritical regime.  Indeed,
\[
r_1=r_2=\frac23,\qquad
q'=-1,\qquad
\kappa=1,\qquad
r_1'=r_2'=-2.
\]
Since \(V_1(x)=V_2(x)=x\) and \(W_u(y)=1-y\), the two published
characteristics coincide and reduce to
\begin{equation}
\mathcal A_6(u)=\mathcal A_7(u)
=
\int_0^1
\left(
\int_x^1(1-y)^{4/3}y^{-2/3}\,\dd y
\right)^{3/2}
x^{-1/2}\,\dd x.
\label{eq:qhalf-A6-explicit}
\end{equation}
The inner integral converges to the finite beta integral
\(\int_0^1(1-y)^{4/3}y^{-2/3}\,\dd y\) as \(x\downarrow0\), so the outer
integrand is \(O(x^{-1/2})\) there.  As \(x\uparrow1\), the inner integral is
\(O((1-x)^{7/3})\), and the outer integrand is
\(O((1-x)^{7/2})\).  Thus \eqref{eq:qhalf-A6-explicit} is finite.  This also
checks boundedness through \eqref{eq:subcritical-A1-norm}, with the negative
powers interpreted by \eqref{eq:algebraic-companion} and
\eqref{eq:ri-algebraic-companion}.

The finite-rank estimate can also be made explicit.  For the uniform
partition of \([c,d]\) into \(m\) pieces,
\[
M_1=M_2=d^{1/2},
\qquad
\omega_1(h)=\omega_2(h)=h^{1/2},
\qquad
U_{c,d}=d-c,
\]
and hence
\[
E_{c,d}(\mathcal P_m)
=
2d^{1/2}(d-c)^{5/2}m^{-1/2}.
\]
Equations \eqref{eq:qhalf-lower-tail}--\eqref{eq:qhalf-upper-tail},
\eqref{eq:truncation-error-estimate}, and
\eqref{eq:uniform-partition-approximation-bound} yield
\begin{align}
a_{m+1}^{(2)}(\mathcal H_2)
&\leq
\Biggl[
\frac23\left(c^{3/2}+1-d^{3/2}\right)\notag\\
&\hspace{18mm}
+\sqrt2\,d^{1/4}(d-c)^{5/4}m^{-1/4}
\Biggr]^2.
\label{eq:qhalf-approximation-rate-window}
\end{align}
Taking \(c_m=m^{-1/6}\) and \(d_m=1-m^{-1/4}\) for sufficiently large
\(m\) shows that
\begin{equation}
a_{m+1}^{(2)}(\mathcal H_2)=O(m^{-1/2}).
\label{eq:qhalf-approximation-rate}
\end{equation}
This calculation exhibits the \(q\)-power addition mechanism at both the
tail and finite-rank levels.
\end{example}

\medskip
\begin{remark}[Interpretation]
The four examples isolate the role of the endpoint conditions. On the
bounded interval with regular weights, both truncated characteristics
vanish explicitly. On the half-line, however, the critical decay
\[
u_3(x)=2(1+x)^{-3}
\]
produces a bounded operator whose upper-tail characteristic does not
vanish.  The asymmetric example shows the same phenomenon in the mixed
range \(p_1=q<p_2\), with a different critical exponent.  Thus the
endpoint-tail requirements constitute genuine compactness conditions rather
than a reformulation of boundedness.  The final example verifies the
nonlocally convex case \(q=1/2\) by both the operator tail test and the deep
subcritical weight characteristic, and it supplies an explicit
approximation-number upper decay estimate.
\end{remark}

\section{Endpoint regimes outside the present framework}
\label{sec:limitations}

The localisation and quantitative theorems cover
\(1<p_1,p_2<\infty\) and \(0<q<\infty\).  Three genuine boundary
phenomena remain outside this framework.

\subsection{\texorpdfstring{The input endpoint \(p_i=1\)}
{The input endpoint pi=1}}

When \(p_i=1\), the conjugate exponent is infinite.  The local dual-weight
integral in Lemma~\ref{lem:local-Hardy-estimates} must then be replaced by an
essential-supremum expression.  In particular, the moduli
\(\omega_i(\delta)\) used to control the sampled Hardy primitives no longer
follow from the same absolute-continuity argument.  A treatment of this
endpoint would therefore require different local hypotheses and a revised
middle-window approximation lemma.

\subsection{\texorpdfstring{The input endpoint \(p_i=\infty\)}
{The input endpoint pi=infinity}}

The input endpoint \(p_i=\infty\) is also outside the standing assumptions.
The finite-\(p_i\) dual-weight formula used in
Lemma~\ref{lem:local-Hardy-estimates} is not the natural local evaluation
mechanism at this endpoint.  There is an additional well-definedness issue
when the interval is left-infinite. Even in the unweighted case, the full
primitive
\[
Hf(x)=\int_{-\infty}^{x}f(t)\,\dd t
\]
need not be finite for an arbitrary \(f\in L^\infty\).  Thus this endpoint
requires separate well-definedness assumptions and a different
strict-middle compactness and approximation argument.

\subsection{\texorpdfstring{The target endpoint \(q=\infty\)}
{The target endpoint q=infinity}}

Uniform absolute continuity on compact subsets of \(L^q\), which drives the
necessity proof for every finite \(q\), fails in \(L^\infty\).  This is not
only a technical defect of the proof. The endpoint-tail criterion itself can
fail for a compact operator.

Indeed, let \(I=(0,1)\), \(p_1=p_2=2\), and \(u=v_1=v_2\equiv1\).  The linear
Hardy map sends the unit ball of \(L^2(0,1)\) to a uniformly bounded and
equicontinuous subset of \(C[0,1]\), since
\[
|Hf(y)-Hf(x)|\leq |y-x|^{1/2}\norm{f}_2.
\]
The Arzel\`a--Ascoli theorem and continuity of pointwise multiplication show
that
\[
\mathcal H_2:L^2(0,1)\times L^2(0,1)\longrightarrow L^\infty(0,1)
\]
is compact.  On the other hand, for \(f=g\equiv1\),
\(\mathcal H_2(f,g)(x)=x^2\), and hence
\[
\norm{\chi_{(d,1)}\mathcal H_2(f,g)}_\infty=1
\qquad(0<d<1).
\]
The general Cauchy--Schwarz bound gives the reverse operator estimate, so
\(\beta(d)=1\) for every \(d<1\).  Thus compactness does not force the
\(L^\infty\) upper-tail norm to vanish.  Any \(q=\infty\) theorem must use a
different compactness invariant, such as equicontinuity, rather than the
finite-\(q\) endpoint-mass mechanism.

\section{Conclusion}
\label{sec:conclusion}

For every bounded weighted bilinear Hardy operator in the stated setting, all
loss of compactness is carried by the two endpoints.  The norm of the
combined endpoint tail has an exact directed limit. It equals the essential
(quasi-)norm, the distance from finite-rank bilinear maps, and the Hausdorff
measure of noncompactness of the joint image of the two unit balls under the
quasi-norm-radius convention.  At the complete-metric scale, the
corresponding \(\vartheta\)-powered quantities coincide exactly.  In
particular, the qualitative localisation theorem and the quantitative
noncompactness theorem are two forms of the same endpoint principle.
The abstract compact-middle and finite-rank tiers isolate this reusable
principle from the interior Hardy realization.

The middle-window argument is constructive.  Piecewise-constant sampling of
the two Hardy primitives produces rank-controlled bilinear maps, with an
error governed by the endpoint tail and the local dual-weight moduli.  This
gives explicit approximation-number estimates.  Once finite-rank distance is
identified with the tail limit, the convergence of the bilinear approximation
numbers to the essential (quasi-)norm follows directly from their defining
rank infima.

Applying prior norm-equivalent boundedness characteristics to truncated
output weights yields weight-only compactness criteria throughout
\(1<p_i<\infty\), \(0<q<\infty\), namely the convex regime, the two mixed
regimes, and the moderate and deep fully subcritical regimes.  The formulas
and the regime split belong to the cited boundedness theory. Their
endpoint-truncated compactness consequences are obtained here.  The four
examples illustrate the separation of boundedness from compactness at sharp
thresholds.  In
particular, the asymmetric choice \(p_1=q=2<p_2=4\) has critical decay
\(r=7/2\), where the operator is bounded but its upper endpoint obstruction
persists.

The formulation covers bounded intervals, both half-lines, and the whole
real line.  In the nonlocally convex range \(0<q<1\), the complete
\(\vartheta\)-metric, rather than local convexity, supports the compactness
argument, while disjoint-support \(q\)-additivity retains the exact tail and
approximation formulas.  The \(q=1/2\) model also shows concretely how the
negative algebraic powers in the deeply subcritical characteristic are
handled.  The input endpoints \(p_i=1\) and \(p_i=\infty\), and the target
endpoint \(q=\infty\), remain outside the present framework for the reasons
recorded in Section~\ref{sec:limitations}.

\section*{Declarations}

\paragraph{Data availability.}
No empirical dataset is used in this mathematical study.

\paragraph{Funding.}
This work received no funding.
\paragraph{Competing interests.}
The author declares no competing interests.

\section*{Generative AI disclosure}
OpenAI's ChatGPT, including ChatGPT Work, and Codex were used as assistive tools during manuscript preparation for mathematical consistency checking, language refinement, readability improvement, and LaTeX preparation. The author independently reviewed the mathematical arguments, references, and final manuscript and takes full responsibility for all statements, proofs, and conclusions.

\appendix
\section{Auxiliary compactness principles in the quasi-Banach range}
\label{app:auxiliary}

The following lemmas isolate the three points at which the proof for
\(0<q<1\) differs from a normed-space argument.

\begin{lemma}[Uniform absolute continuity on compact subsets of \(L^q\)]
\label{lem:uniform-absolute-continuity}
Let \(0<q<\infty\), and let \(K\) be a compact subset of \(L^q(u;I)\).  If
\(E_n\subset I\) are measurable and
\[
\chi_{E_n}(x)\longrightarrow0
\quad\text{for \(u(x)\,dx\)-almost every }x\in I,
\]
then
\begin{equation}
\lim_{n\to\infty}
\sup_{h\in K}\norm{\chi_{E_n}h}_{L^q(u;I)}=0.
\label{eq:uniform-absolute-continuity}
\end{equation}
\end{lemma}

\begin{proof}
Let \(\vartheta=\min\{1,q\}\), fix \(\varepsilon>0\), and put
\[
\eta:=\frac{\varepsilon}{2^{1/\vartheta}}.
\]
Compactness in the metric \eqref{eq:Lq-metric} implies total boundedness in
the quasi-norm topology.  Hence there are \(h_1,\ldots,h_N\in K\) such that
\[
K\subset\bigcup_{j=1}^{N}B_Y(h_j,\eta).
\]
For every \(j\),
\[
\norm{\chi_{E_n}h_j}_Y^q
=
\int_I\chi_{E_n}(x)|h_j(x)|^q u(x)\,dx
\longrightarrow0
\]
by \eqref{eq:indicator-tail-pointwise}, equivalently by dominated
convergence.  Since the set of centres is finite, there is
\(n_0\) such that
\[
\norm{\chi_{E_n}h_j}_Y<\eta
\]
for all \(j\) and \(n\geq n_0\).

Given \(h\in K\), choose \(j\) with \(\norm{h-h_j}_Y<\eta\).  Multiplication
by an indicator is contractive on \(Y\).  Thus, by
\eqref{eq:theta-triangle},
\begin{align*}
\norm{\chi_{E_n}h}_Y^{\vartheta}
&\leq
\norm{\chi_{E_n}(h-h_j)}_Y^{\vartheta}
+\norm{\chi_{E_n}h_j}_Y^{\vartheta}\\
&<2\eta^{\vartheta}=\varepsilon^{\vartheta}
\end{align*}
for \(n\geq n_0\).  Taking the supremum over \(h\in K\) proves
\eqref{eq:uniform-absolute-continuity}.  Notice that no convexity of the unit
ball is used.
\end{proof}

\begin{lemma}[Finite-dimensional ranges]
\label{lem:finite-dimensional-range}
Let \(Y\) be a quasi-normed linear space whose quasi-norm satisfies
\eqref{eq:theta-triangle} for some \(0<\vartheta\leq1\).  Every bounded subset
of a finite-dimensional subspace of \(Y\) has compact closure.  Consequently,
every bounded finite-rank bilinear map into \(Y\) is compact.
\end{lemma}

\begin{proof}
Let \(E=\operatorname{span}\{e_1,\ldots,e_m\}\subset Y\), with the displayed
vectors linearly independent, and define
\[
J:\mathbb K^m\longrightarrow E,
\qquad
J(a_1,\ldots,a_m):=\sum_{j=1}^m a_j e_j,
\]
where \(\mathbb K\) is \(\mathbb R\) or \(\mathbb C\).  Iterating
\eqref{eq:theta-triangle} gives
\[
\norm{J(a)}_Y^{\vartheta}
\leq
\sum_{j=1}^m |a_j|^{\vartheta}\norm{e_j}_Y^{\vartheta},
\]
so \(J\) is continuous.  Also,
\[
\left|\norm{x}_Y^{\vartheta}-\norm{y}_Y^{\vartheta}\right|
\leq\norm{x-y}_Y^{\vartheta},
\]
which follows by applying \eqref{eq:theta-triangle} in both directions.
Therefore \(a\mapsto\norm{J(a)}_Y\) is continuous.  It is strictly positive
on the Euclidean unit sphere, which is compact, and hence has a positive
minimum there.  Homogeneity now shows that \(J^{-1}\) is continuous.  Thus
the quasi-norm topology on \(E\) is the ordinary finite-dimensional topology.
If \(S\subset E\) is bounded, then \(J^{-1}(S)\) is Euclidean-bounded. The
image under \(J\) of its Euclidean closure is a compact subset of \(Y\) that
contains the closure of \(S\).  Hence bounded subsets of \(E\) have compact
closure.

If a bounded bilinear map \(T\) has range in \(E\), then
\(T(B_{X_1}\times B_{X_2})\) is bounded in \(E\), so the preceding conclusion
shows that its closure is compact in \(Y\).
\end{proof}

\begin{lemma}[Operator-quasi-norm closure of compact bilinear maps]
\label{lem:closure-compact-bilinear}
Let \(X_1,X_2\) be normed spaces, and let \(Y\) be a complete quasi-normed
space satisfying \eqref{eq:theta-triangle} for some
\(0<\vartheta\leq1\).  Suppose that \(T_n:X_1\times X_2\to Y\) are compact
bilinear maps and
\[
\norm{T_n-T}_{X_1\times X_2\to Y}\longrightarrow0
\]
for a bounded bilinear map \(T:X_1\times X_2\to Y\).  Then \(T\) is compact.
\end{lemma}

\begin{proof}
Let \(B_i\) be the closed unit ball of \(X_i\), fix \(\varepsilon>0\), and set
\(\eta=\varepsilon/2^{1/\vartheta}\).  Choose \(n\) so that
\[
\norm{T-T_n}_{X_1\times X_2\to Y}<\eta.
\]
Since \(T_n\) is compact, \(T_n(B_1\times B_2)\) has a finite
\(\eta\)-net \(y_1,\ldots,y_N\).  For \(x_i\in B_i\), choose \(j\) such that
\(\norm{T_n(x_1,x_2)-y_j}_Y<\eta\).  Then
\begin{align*}
\norm{T(x_1,x_2)-y_j}_Y^{\vartheta}
&\leq
\norm{(T-T_n)(x_1,x_2)}_Y^{\vartheta}
+\norm{T_n(x_1,x_2)-y_j}_Y^{\vartheta}\\
&<2\eta^{\vartheta}=\varepsilon^{\vartheta}.
\end{align*}
Thus \(T(B_1\times B_2)\) is totally bounded in the metric induced by the
quasi-norm.  Completeness of \(Y\) implies that its closure is compact.
\end{proof}

\end{document}